\documentclass[11pt,a4paper]{amsart}
\usepackage{amsfonts}

\usepackage{}
\usepackage[mathscr]{eucal}
\usepackage{amssymb}
\usepackage{amsbsy}

\usepackage{CJK,CJKnumb}
\usepackage[CJKbookmarks,colorlinks,
            linkcolor=black,
            anchorcolor=black,
            citecolor=blue]{hyperref}
\usepackage{color,soul}              
\usepackage{indentfirst}        
\usepackage{latexsym,bm}        
\usepackage{graphics,graphicx}
\usepackage{cases}
\usepackage{pifont}
\usepackage{txfonts}
\usepackage{xcolor}
\usepackage[all,knot,poly]{xy}
\usepackage[usenames,dvipsnames]{pstricks}
\usepackage{epsfig}
\usepackage{pst-grad} 
\usepackage{pst-plot} 

\allowdisplaybreaks[4]  

\numberwithin{equation}{section}

\newcommand{\al}{\alpha}
\newcommand{\be}{\beta}
\newcommand{\de}{\delta}
\newcommand{\De}{\Delta}
\newcommand{\ep}{\epsilon}
\newcommand{\ve}{\varepsilon}
\newcommand{\emp}{\emptyset}
\newcommand{\ga}{\gamma}

\newcommand{\la}{\lambda}

\newcommand{\La}{\Lambda}
\newcommand{\ot}{\otimes}

\newcommand{\Om}{\Omega}
\newcommand{\si}{\sigmaup}
\newcommand{\te}{\theta}
\newcommand{\Te}{\Theta}
\newcommand{\vt}{\vartheta}

\newcommand{\mB}{\mathcal{B}}

\newcommand{\mG}{\mathcal{G}}
\newcommand{\mH}{\mathcal{H}}
\newcommand{\mK}{\mathcal{K}}

\newcommand{\mC}{\mathscr{C}}
\newcommand{\pp}{\mathscr{P}}

\newcommand{\fh}{\mathfrak{h}}
\newcommand{\fg}{\mathfrak{g}}
\newcommand{\fq}{\mathfrak{q}}
\newcommand{\fd}{\mathfrak{D}}
\newcommand{\fs}{\mathfrak{S}}
\newcommand{\fr}{\mathfrak{R}}

\newcommand{\bN}{\mathbb{N}}
\newcommand{\bQ}{\mathbb{Q}}

\newcommand{\bZ}{\mathbb{Z}}
\newcommand{\rc}{\rotatebox{90}{$\circlearrowright$}}

\newcommand{\bk}{\mathbb{K}}
\newcommand{\mb}[1]{\mbox{#1}}

\newcommand{\bs}[1]{{\scriptsize\mbox{#1}}}

\newcommand{\lan}{\langle}
\newcommand{\ran}{\rangle}
\newcommand{\lb}{\left(}
\newcommand{\rb}{\right)}
\newcommand{\rw}{\rightarrow}
\newcommand{\beq}{\begin{equation}}
\newcommand{\eeq}{\end{equation}}

\begin{document}

\newtheorem{theorem}{Theorem}[subsection]

\newtheorem{lem}[theorem]{Lemma}

\newtheorem{cor}[theorem]{Corollary}
\newtheorem{prop}[theorem]{Proposition}

\theoremstyle{remark}
\newtheorem{rem}[theorem]{Remark}

\newtheorem{defn}[theorem]{Definition}

\newtheorem{exam}[theorem]{Example}

\theoremstyle{conjecture}
\newtheorem{con}[theorem]{Conjecture}

\renewcommand\arraystretch{1.2}

\title[Representation theory of 0-Hecke-Clifford algebras]{Representation theory of 0-Hecke-Clifford algebras}

\author[Li]{Yunnan Li}
\address{School of Mathematics, South China University of Technology, Guangzhou 510640, China}
\email{scynli@scut.edu.cn}


\begin{abstract}
The representation theory of 0-Hecke-Clifford algebras as a degenerate case is not semisimple and also with rich combinatorial meaning. Bergeron et al. have proved that the Grothendieck ring of the category of finitely generated supermodules of 0-Hecke-Clifford algebras is isomorphic to the algebra of peak quasisymmetric functions defined by Stembridge. In this paper we further study the category of finitely generated projective supermodules and clarify the correspondence between it and the peak algebra of symmetric groups. In particular, two kinds of restriction rules for induced projective supermodules are obtained. After that, we consider the corresponding Heisenberg double and its Fock representation to prove that the ring of peak quasisymmetric functions is free over the subring of symmetric functions spanned by Schur's Q-functions.

\end{abstract}

\maketitle

\section{Introduction}

As a $q$-deformation of the Sergeev algebra, the Hecke-Clifford algebra $HCl_n(q)$ is defined by G. Olshanski in \cite{Ols} mixing the Hecke algebra $H_n(q)$ and the Clifford algebra $Cl_n$. When $q$ is generic, it satisfies the Schur-Sergeev-Olshanski super-duality with the quantum enveloping algebra of queer Lie superalgebras $\fq_n$. Moreover, the Grothendieck group of the tower of Hecke-Clifford algebras is isomorphic to the subalgebra of symmetric functions spanned by Schur's Q-functions, parallel to the classical case of Sergeev algebras ($q=1$) \cite[\S 3.3]{CW}. For the root of unity case, Brundan and Kleshchev in \cite{BK1} consider the (affine) Hecke-Clifford algebra and relate its representation category with the positive part of the universal enveloping algebra for the affine Kac-Moody algebra $\fg=A_{2l}^{(2)}$. Recently, Mori in \cite{Mor} extends the method of cellular algebras to the superalgebra setting to study the cellular representation theory of Hecke-Clifford algebras uniformly only requiring that $q$ is invertible. In a word, the representation theory of (affine) Hecke-Clifford algebras has been widely studied; see also \cite{HKS},\cite{Wan},\cite{WW}, etc.

There still leaves a degenerate case when $q=0$ and also with nice combinatorial aspect. The 0-Hecke-Clifford algebra $HCl_n(0)$ is first considered by Bergeron et al. in \cite{BHT}, where they mainly construct the simple supermodules of $HCl_n(0)$, and prove the Frobenius isomorphism between the Grothendieck group of the category of finitely generated supermodules of 0-Hecke-Clifford algebras and the Stembridge algebra of peak quasisymmetric functions. Since the cellular approach fails in this degenerate case, one needs different techniques to handle its representation theory. In this paper, we dually discuss the category of finitely generated projective supermodules.
Note that the graded Hopf dual of the algebra of peak quasisymmetric functions is the peak algebra of symmetric groups \cite{Sch}. Here we confirm the Hopf dual pair between the Grothendieck groups of the above two supermodule categories and explicitly relate the projective one to the peak algebra. In fact, there already have many nice papers considering towers of algebras together with their corresponding Grothendieck groups, which provide a bunch of (combinatorial) Hopf algebras; see \cite{BK1},\cite{CL},\cite{FJW},\cite{KT},\cite{SY}, etc.

Recently, Berg et al. in \cite{BBSSZ} construct a noncommutative lift of Schur functions, called the immaculate basis, and then in \cite{BBSSZ1} find the correspondence of its dual basis to the category of finitely generated modules of 0-Hecke algebras under the Frobenius isomorphism defined in \cite{KT}. Inspired by their work, we construct a lift of Schur's Q-functions to the peak algebra and thus extract a new basis from it in \cite{JL}. Dually we define a new basis, called the quasisymmetric Schur's Q-functions, in the Stembridge algebra, whose expansion in the peak functions is expected to be positive based on concrete examples \cite[Conjecture 4.15]{JL}. It implies the chance for a representation theoretical meaning of such basis on the supermodule category of 0-Hecke-Clifford algebras. Burying such target in mind, we want to further study the representation theory of 0-Hecke-Clifford algebras.

As an application, we also use the corresponding Heisenberg double and its Fock representation from the tower of 0-Hecke-Clifford algebras to prove that the ring of peak quasisymmetric functions is free over the subring of symmetric functions spanned by Schur's Q-functions. Such method is purposed by Savage et al. in \cite{SY} to give a new proof of the freeness of the ring of quasisymmetric functions over the ring of symmetric functions. For further discussion of twisted version of Heisenberg doubles, one can refer to \cite{RS1}, \cite{RS}.

The organization of the paper is as follows. In $\S 2$ we provide some notation, definitions for all combinatorial Hopf algebras that we involve. Some preliminaries on superalgebras and the terminology of towers of superalgebras are recalled. In $\S 3$ we focus on the representation theory of 0-Hecke-Clifford algebras. Their projective supermodules induced from those of 0-Hecke algebras are mainly considered. Finally we clarify a dual pair of graded Hopf algebra structures on two Grothendieck groups of finitely generated supermodules and projective supermodules of 0-Hecke-Clifford algebras, whose Frobenius superalgebra structures are also figured out. In $\S 4$ we further give the concrete relation between the Grothendieck groups defined in the previous section with the peak algebra of symmetric groups and its dual, extending the results in \cite{BHT}. In particular for those induced projective supermodules, two restriction rules and the decomposition formula to indecomposable ones are given. As a final application, the freeness of the ring of peak quasisymmetric functions over the subring of symmetric functions spanned by Schur's Q-functions is proved.

\section{Preliminaries}
\subsection{Notation and definitions}
Throughout this paper, we work over an algebraically closed field $\bk$ of characteristic 0 for simplicity. Denote by $\bN$ (resp. $\bN_0$) the set of positive (resp. nonnegative) integers. Given any $m,n\in\bN$, let $[m,n]:=\{m,m+1,\dots,n\}$ if $m\leq n$ and $\emp$ otherwise. Also $[n]:=[1,n]$ for short. Let $2^{[n]}$ be the set of subsets of $[n]$ and $\mC(n)$ be the set of compositions of $n$, consisting of ordered tuples of positive integers summed up to $n$.  We denote $\al\vDash n$ when $\al\in \mC(n)$. Let $\mC:=\bigcup\limits_{n\geq1}^.\mC(n)$. Given $\al=(\al_1,\dots,\al_r)\vDash n$, let $\ell(\al)=r$ be its length and define its associated \textit{descent set} as
\[D(\al):=\{\al_1,\al_1+\al_2,\dots,\al_1+\cdots+\al_{r-1}\}
\subseteq[n-1].\]
Given a permutation $w=w_1\cdots w_n\in\fs_n$, we also define its associated \textit{descent set} as
\[D(w):=\{i\in[n-1]:w_i>w_{i+1}\}\]
and \textit{descent composition} $c(w)\vDash n$ such that $D(c(w))=D(w)$. Note that $\mC(n)\rw 2^{[n-1]},\,\al\mapsto D(\al)$ is a bijection. The refining order $\leq$ on $\mC(n)$ is defined by
\[\al\leq\be\mb{ if and only if }D(\be)\subseteq D(\al),\,\forall\al,\be\vDash n.\]
In general, for $\al\in\bN^r_0$, let $\ell(\al)=|\{i\,:\,\al_i>0\}|$.
Given $\al\vDash n$, the corresponding \textit{descent class} of the symmetric group $\fs_n$ is defined by
\[\fd_\alpha:=\{w\in\fs_n:D(w)=D(\alpha)\}.\]
For $\al=(\al_1,\dots,\al_r)\vDash n$, define its three counterparts:

(1) the \textit{reverse} of $\al$, $\bar{\al}:=(\al_r,\dots,\al_1)$,
such that $D(\bar{\al})=\{i\in[n-1]:n-i\in D(\al)$.

(2) the \textit{complement} of $\al$, $\al^c\vDash n$
such that $D(\al^c)=[n-1]\backslash D(\al)\}$.

(3) the \textit{conjugation} of $\al$, $\al^*:=\bar{\al}^c$.

We also recall the concept of peaks. A subset $P\subseteq[n]$ is called a \textit{peak set} in $[n]$ if $P\subseteq[2,n-1]$ and $i\in P\Rightarrow i-1\notin P$. Denote by $\pp_n$ the collection of peak sets in $[n]$, $\pp:=\bigcup\limits^._{n\geq1}\pp_n$, and $\emp_n$ the empty set $\emp$ in $\pp_n$. Given $\al=(\al_1,\dots,\al_r)\vDash n$, let
\[P(\al):=\{x\in[2,n-1]:x-1\notin D(\al),\,x\in D(\al)\}\]
be its associated peak set in $[n]$, while its associated \textit{valley set} $V(\al)\subseteq[n]$ is defined by
\[V(\al)\cap[2,n]=\{x\in[2,n]:x-1\in D(\al),\,x\notin D(\al)\}\]
and $1\in V(\al)\Leftrightarrow 1\notin D(\al)$. Note that $|V(\al)|=|P(\al)|+1$. Given a permutation $w=w_1\cdots w_n\in\fs_n$, its \textit{peak set} \[P(w):=\{i\in[2,n-1]\,:\,w_{i-1}<w_i>w_{i+1}\}=P(c(w)).\].


Define the algebra of \textit{noncommutative symmetric functions} as the free associative $\bk$-algebra generated by the symbols $H_n\,(n\in\mathbb{N})$  and denote it by NSym \cite{GKL}. Define another set of generators $E_n\,(n\in\mathbb{N})$ in NSym by
\[\sum_{i=0}^n(-1)^iE_iH_{n-i}=\de_{n,0}.\]
The algebra
$\mb{NSym}=\bigoplus_{n=0}^{\infty} \mb{NSym}_n$ is a $\mathbb Z$-graded algebra
under the gradation given by $\mb{deg}(H_n)=n$, where $\mb{NSym}_n$ is
the subspace of homogeneous elements of degree $n$.
Let \[H_\al:=H_{\al_1}\cdots H_{\al_r},E_\al:=E_{\al_1}\cdots E_{\al_r},~\al=(\al_1,\dots,\al_r)\vDash n.\]
If we change base from $\bk$ to $\bZ$, then both $\{H_\al\}_{\al\vDash n}$ and $\{E_\al\}_{\al\vDash n}$ are $\bZ$-bases of $\mb{NSym}_n$, called the {\it noncommutative complete} and {\it elementary symmetric functions} respectively. There exists another important $\bZ$-basis $\{R_\al\}_{\al\vDash n}$ of $\mb{NSym}_n$, called the \textit{noncommutative ribbon Schur functions} and are defined by
\[R_\al:=\sum_{\be\geq\al}(-1)^{l(\be)-l(\al)}H_\be.\]
With the coproduct defined by
\beq\label{coq}
\De(H_n)=\sum_{k=0}^n H_k\ot H_{n-k},\eeq
NSym becomes a graded and connected Hopf algebra.
Moreover, the graded Hopf dual of NSym is the algebra of \textit{quasisymmetric functions}, denoted by QSym \cite{MR}.
It is a subring of the power series ring $\bk[[x_1,x_2,\dots]]$ in the commuting variables $x_1,x_2,\dots$ and has a linear basis, the \textit{monomial quasisymmetric functions}, defined by
\[M_\al:=M_\al(x)=\sum\limits_{i_1<\cdots< i_r}x_{i_1}^{\al_1}\cdots x_{i_r}^{\al_r},\]
where $\al=(\al_1,\dots,\al_r)$ varies over the set $\mC$ of compositions. There is another important basis, the \textit{fundamental quasisymmetric functions}, defined by
\[F_\al:=F_\al(x)=\sum\limits_{i_1\leq\cdots\leq i_n\atop i_k<i_{k+1}\mb{ \tiny if }k\in D(\al)}x_{i_1}\cdots x_{i_n},\,\al\vDash n.\]
In other words, $F_\al=\sum_{\be\leq\al}M_\be$. Meanwhile, the canonical pairing $\lan\cdot,\cdot\ran$ between NSym and QSym is defined by
\[\lan H_\al,M_\be\ran=\lan R_\al,F_\be\ran=\de_{\al,\,\be}\]
for any $\al,\be\in\mC$.

Let $\La$ be the graded ring of symmetric functions in the commuting variables $x_1,x_2,\dots$, with integer coefficients, and $\Om$ be the subring of $\La$ generated by the symmetric functions $q_n\,(n\geq1)$, which are defined by
\[\sum_{n\geq0}q_nz^n=\prod_{i\geq1}\dfrac{1+x_iz}{1-x_iz}.\]
Note that $\Om_\bQ:=\Om\ot_\bZ\bQ$ is isomorphic to $\bQ[p_{2k-1}:k\in\bN]$, where $p_n$'s are the power-sum symmetric functions.
It has the following canonical inner product $[\cdot,\cdot]$ defined by
\[[p_\la,p_\mu]=z_\la2^{-\ell(\la)}\de_{\la,\,\mu}\]
for any partitions $\la,\mu$ only with odd parts, where $z_\la:=\prod_{i\geq1}m_i!i^{m_i}$ for $\la=(1^{m_1},2^{m_2},\dots)$ (see \cite[Ch. III, \S 8]{Mac}). Now introduce two Hopf algebra epimorphisms. One is
\[\te:\La\rw\Om,\quad h_n\mapsto q_n,\,n\geq1,\]
such that $\te(p_n)=(1-(-1)^n)p_n,\,n\geq1$, and the other is
\[\pi:\mb{NSym}\rw\La,\,H_n\mapsto h_n,\]
called the \textit{forgetful map} and satisfying
\[\lan F,f\ran=\lan\pi(F),f\ran,\, F\in\mb{NSym},f\in\La.\]

\subsection{The peak subalgebra and its Hopf dual}
Introduce
\[Q_0:=1,\,Q_n:=\sum_{k=0}^nE_kH_{n-k},\,n\geq1\]
and also $Q_n:=0,\,n<0$ for convenience.
Let $Q_\al:=Q_{\al_1}\cdots Q_{\al_r},~\al=(\al_1,\dots,\al_r)\vDash n$. Then $Q_n\,(n\geq1)$ satisfy the following \textit{Euler relations} (also called the \textit{generalized Dehn-Sommerville relation})
\beq\label{np}\sum_{r+s=n}(-1)^rQ_rQ_s=0,\,\forall n\geq1,\eeq
or equivalently,
\beq\label{eo}\sum_{i=1}^{n-1}(-1)^{i-1}Q_iQ_{n-i}=\begin{cases}
2Q_n, & $if $ n $ is even$,\\
0, & $if $ n $ is odd$.
\end{cases}\eeq
Let Peak be the Hopf subalgebra of NSym generated by $Q_n\,(n\geq1)$. Then $\mb{Peak}_n:=\mb{Peak}\cap\mb{NSym}_n$ is isomorphic to the \textit{peak algebra} of the symmetric group $\mathfrak{S}_n$ when endowed with the internal product \cite{BHT,Sch}. For any $P\in\pp_n$, define
\beq\label{pi}\Xi_P:=\sum_{P(\al)=P}R_\al\in\mb{NSym}.\eeq
Then $\{\Xi_P\}_{P\in\pp_n}$ forms a linear basis of $\mb{Peak}_n$ \cite[\S 2]{BHT}. Note that by \cite[Eq.(6)]{BHT},
\beq\label{gen}Q_n=2\Xi_{\emp_n}=2\sum_{k=0}^{n-1}R_{1^k,n-k},\,n\geq1.\eeq
There exists a surjective Hopf algebra homomorphism
\[\Te:\mb{NSym}\rw\mb{Peak},\quad H_n\mapsto Q_n,\,n\geq1,\]
called the \textit{descent-to-peak transform} \cite[\S 5]{KLT}, where $\mb{Ker }\Te$ is the Hopf ideal of NSym generated by
$\sum_{r+s=n}(-1)^rH_rH_s,\,n\geq1$ \cite[Theorem 5.4]{BMSW}.

Next we introduce the graded Hopf dual of Peak, the \textit{Stembridge algebra} $\mb{Peak}^*$ of peak quasisymmetric functions, defined in \cite{Ste}. This is a Hopf subalgebra of QSym, with a natural basis called the Stembridge's \textit{peak functions}. They can be defined by \cite[Prop. 3.5]{Ste}
\beq\label{ka}K_P:=2^{|P|+1}\sum_{\al\vDash n\atop P\subseteq D(\al)\triangle(D(\al)+1)}F_\al,\,P\in\pp_n,\eeq
where $D\triangle(D+1)=D\backslash(D+1)\cup (D+1)\backslash D$ for any $D=\{D_1<\cdots<D_r\}\subseteq[n-1]$ and $D+1:=\{x+1:x\in D\}$.
By \cite[Prop. 2.2]{Ste}, we also have
\[K_P=\sum_{\al\vDash n\atop P\subseteq D(\al)\cup(D(\al)+1)}2^{\ell(\al)}M_\al,\,P\in\pp_n\]
and in particular, by \cite[(2.5)]{Ste},
\beq\label{ke}
K_{\emp_n}=q_n=2\sum_{\al\vDash n}F_\al=\sum_{\al\vDash n}2^{\ell(\al)}M_\al.\eeq
There also exists a surjective Hopf algebra homomorphism
\[\vt:\mb{QSym}\rw\mb{Peak}^*,\quad F_\al\mapsto K_{P(\al)},\]
called the \textit{descent-to-peak map}.

Note that Peak can be regarded as a noncommutative lift of $\Om$. The following commutative diagrams illustrate the situation.
\beq\label{co}\xymatrix@=2em{\mb{NSym}\ar@{->}[r]^-{\Te}\ar@{->}[d]_-{\pi}
&\mb{Peak}\ar@{->}[d]^-{\pi}\\
\La\ar@{->}[r]^-{\te}&\Om},\quad
\xymatrix@=2em{\mb{QSym}\ar@{->}[r]^-{\vt}
&\mb{Peak}^*\\
\La\ar@{->}[r]^-{\te}\ar@{->}[u]&\Om\ar@{->}[u]},\eeq
where the vertical maps in the second diagram are inclusions. Also, the graded Hopf dual pairing between Peak and $\mb{Peak}^*$ is defined by
\beq\label{pp}[\cdot,\cdot]:\mb{Peak}\times\mb{Peak}^*\rw\bk,\quad [\Xi_P,K_Q]=\de_{P,Q},\,P,Q\in\pp,\eeq
which satisfies the following property \cite[Cor. 5.6.]{Sch},
\beq\label{pro}\lan\Te(F),f\ran=\lan F,\vt(f)\ran=[\Te(F),\vt(f)],\, F\in\mb{NSym},f\in\mb{QSym}.\eeq

\subsection{Preliminaries on superalgebras}

We first recall some standard results about the
representation theory of finite dimensional (associative) superalgebras, referring to \cite{BK1,BK}. A \textit{superalgebra} $A$ over $\bk$ is a $\bZ_2$-graded $\bk$-vector space $A=A_{\bar{0}}\oplus A_{\bar{1}}$ which is also an algebra such that $A_iA_j\subseteq A_{i+j},\,i,j\in\bZ_2$. Any superalgebra $A$ considered here has the unit $1=1_A\in A_{\bar{0}}$. Given $a\in A_i\,(i\in\bZ_2)$, let the \textit{degree} of $a$ be $|a|:=i$.  Homomorphisms between two superalgebras are usual algebra homomorphisms.
For two $\bZ_2$-graded $\bk$-vector spaces $V,W$, $\mb{Hom}_{\bk}(V,W)$ is $\bZ_2$-graded such that $f\in\mb{Hom}_\bk(V,W)_i$ if $f(V_j)\subseteq W_{i+j},\,i,j\in\bZ_2$. The base field $\bk$ serves as a one dimensional even space.

A left $A$-\textit{supermodule} $M$ is a $\bZ_2$-graded $\bk$-vector space $M=M_{\bar{0}}\oplus M_{\bar{1}}$ which is also an $A$-module such that $A_iM_j\subseteq M_{i+j},\,i,j\in\bZ_2$. Given $m\in M_i\,(i\in\bZ_2)$, also let the \textit{degree} of $m$ be $|m|:=i$.
A \textit{morphism} $f$ between two left $A$-supermodules $M,N$ is a linear map such that $f(am)=(-1)^{|f||a|}af(m),\,a\in A,m\in M$. We denote the $\bZ_2$-graded $\bk$-vector space of all such morphisms by $\mb{Hom}_A(M,N)$. On the other hand, $M^*:=\mb{Hom}_\bk(M,\bk)$ has the following natural left $A$-supermodule structure:
\[(af)(m)=(-1)^{|a|(|f|+|m|)}f(ma),\,a\in A,f\in M^*,m\in M,\]
if $M$ is a right $A$-supermodule, and the natural right $A$-supermodule structure:
$(fa)(m)=f(am)$, if $M$ is a left $A$-supermodule.

An $A$-supermodule is \textit{irreducible} (or simple) if it is non-zero and has no non-zero proper super submodules. Then it is either irreducible as an ordinary $A$-module (called \textit{of type M}) or else reducible as an $A$-module (called \textit{of type Q}).

All the finitely generated $A$-supermodules together with their morphisms constitute a \textit{superadditive} category, denoted by $A$-mod. Besides, all the finitely generated projective $A$-supermodules form a full subcategory of $A$-mod, denoted by $A$-pmod. There exists the \textit{parity change functor} $\Pi:A\mb{-mod}\rw A\mb{-mod}$ such that $\Pi M$ is the same underlying vector space but with $(\Pi M)_i=M_{i+1}\,(i\in\bZ_2)$ and the following new action:
\[a.m:=(-1)^{|a|}am,\,a\in A,m\in M.\]
Note that one can identify $\mb{Hom}_A(M,N)_i$ with $\mb{Hom}_A(M,\Pi N)_{i+1}$ for any $i\in\bZ_2$, and obviously the map $M\rw\Pi M,\,m\mapsto(-1)^{|m|}m$ is a homomorphism of usual $A$-modules.

Suppose that $\nu$ is an automorphism of a superalgebra $A$.
For any $M\in A$-mod, we can twist the left action on $M$ with $\nu$ to define the following twisted left $A$-module, denoted by ${^\nu M}$:
\[a._\nu m:=\nu(a)m,\,m\in M,a\in A.\]
If $\nu$ is an \textit{unsigned} anti-automorphism of $A$, i.e. $\nu(ab)=\nu(b)\nu(a),\,a,b\in A$, then we can twist the right action of $A$ on $M^*$ with $\nu$ to define the following twisted left $A$-module, also denoted by $^\nu(M^*)$:
\[(a._\nu f)(m):=f(\nu(a)m),\,f\in M^*,m\in M,a\in A.\]
For any right $A$-supermodule $M$, we use the notation $M^\nu$ instead for the right twisted module structure without confusion.

The category $A$-mod has the \textit{underlying even subcategory} with the same objects but only \textit{even} morphisms, which is an abelian category. Hence, we can
define the corresponding \textit{Grothendieck group} $K_0(A\mb{-mod})$ to be the quotient of the free abelian group with all objects
in $A\mb{-mod}$ as a basis by the subgroup generated by

(1) $M_1-M_2+M_3$ for every short exact sequence $0\rw M_1\rw M_2\rw M_3\rw0$ in the underlying even subcategory.

(2) $M-\Pi M$ for every $M\in A\mb{-mod}$.

Similarly we define the Grothendieck group $K_0(A\mb{-pmod})$. For any $M\in A$-mod or $A$-pmod, its class in such Grothendieck group is denoted by $[M]$. Note that the natural embedding from $A\mb{-pmod}$ to $A\mb{-mod}$ induces the \textit{Cartan map}
\beq\label{car}\chi:K_0(A\mb{-pmod})\rw K_0(A\mb{-mod}),\eeq
which describes the multiplicity of composition factors in projective modules. Also there is a canonical embedding
\[K_0(A\mb{-mod})\ot_\bZ K_0(B\mb{-mod})\rw K_0((A\ot B)\mb{-mod}),\,[M]\ot[N]\rw[M\ot N],\]
which is similarly defined on pmod's and an isomorphism when changing the base ring to $\bQ$.

There is a natural bilinear form
\beq\label{pair}\lan\cdot,\cdot\ran:K_0(A\mb{-pmod})\times
K_0(A\mb{-mod})\rw\bZ,\,\lan[P],[M]\ran
=\mb{dim}_\bk\mb{Hom}_A(P,M).\eeq
For the pair on $K_0((A\ot B)\mb{-pmod})\times
K_0((A\ot B)\mb{-mod})\rw\bZ$, we have
\beq\label{tensor}\lan[P\ot Q],[M\ot N]\ran=\lan[P],[M]\ran
\lan[Q],[N]\ran.\eeq
\begin{rem}
The pair defined in (\ref{pair}) only involves dimensions of homomorphism spaces but not their graded dimensions as in \cite{RS}.
It makes the critical difference since the authors in \cite{RS} consider the structure of twisted dual Hopf algebras indeed.
\end{rem}

If $A$ is a finite dimensional superalgebra, let $V_1,\dots,V_r$ be a complete list of non-isomorphic simple $A$-supermodules. If $P_i$ is the projective cover of $V_i$ in $A$-mod for $i=1,\dots,r$, then $P_1,\dots,P_r$ is a complete list of non-isomorphic indecomposable projective $A$-supermodules and we have
\[K_0(A\mb{-mod})=\bigoplus_{i=1}^r\bZ[V_i],\,K_0(A\mb{-pmod})=\bigoplus_{i=1}^r\bZ[P_i].\]
Note that
\beq\label{eva}\lan[P_i],[M_j]\ran=\begin{cases}
1,& \mb{if }i=j\mb{ and }M_i\mb{ is of type M},\\
2,& \mb{if }i=j\mb{ and }M_i\mb{ is of type Q},\\
0,& \mb{otherwise}.
\end{cases}\eeq

Given two superalgebras $A$ and $B$, the tensor product $A\ot B$ has the superalgebra structure defined by the following twisted multiplication:
\[(a\otimes b)(c\otimes d)=(-1)^{|b||c|}ac\otimes bd,\,a,c\in A,b,d\in B\]
and $|a\ot b|:=|a|+|b|$ for any homogeneous $a\in A,b\in B$.
Given an $A$-supermodule $M$ and a $B$-supermodule $N$, the tensor product $M\ot N$ has the $A\ot B$-supermodule structure defined by:
\[(a\otimes b)(m\otimes n)=(-1)^{|b||m|}am\otimes bn,\,a\in A,b\in B,m\in M,n\in N\]
and $|m\ot n|:=|m|+|n|$ for any homogeneous $m\in M,n\in N$.

\begin{lem}[{\cite[\S 2]{BK}}]\label{tq}
If $M$ is a finite dimensional irreducible $A$-supermodule of type Q, then there exist bases $\{v_1,\dots,v_n\}$ for $M_{\bar{0}}$ and $\{\bar{v}_1,\dots,\bar{v}_n\}$ for $M_{\bar{1}}$ such that
$\mb{span}_\bk\{v_1+\bar{v}_1,\dots,v_n+\bar{v}_n\}$ and  $\mb{span}_\bk\{v_1-\bar{v}_1,\dots,v_n-\bar{v}_n\}$
form two non-isomorphic irreducible $A$-modules. Moreover, the linear map $J_M\in\mb{End}_\bk(M)$ defined by $v_i\mapsto \bar{v}_i,\,\bar{v}_i\mapsto-v_i$ is an odd $A$-supermodule automorphism of $M$.
\end{lem}

\begin{lem}[Schur's lemma]\label{Schur}
If $M$ is a finite dimensional irreducible $A$-supermodule, then
\[\mb{End}_A(M)=\begin{cases}
\mb{span}_\bk\{\mb{id}_M\},& \mb{if }M\mb{ is of type M},\\
\mb{span}_\bk\{\mb{id}_M,J_M\},& \mb{if }M\mb{ is of type Q},
\end{cases}\]
where $J_M$ is as in Lemma \ref{tq}.
 \end{lem}


\subsection{Towers of superalgebras}
Now mainly for 0-Hecke-Clifford algebras, we introduce the following important definition (see \cite[Def. 4.1]{RS}) which extends the original one in \cite[\S 3]{BL}. Further discussion can be found in \cite{BLL}.
\begin{defn}
Let $A=\bigoplus_{n\geq0}A_n$ be a graded superalgebra over $\bk$ with multiplication $\mu:A\otimes A\rw A$. Then $A$ is called a \textit{tower of superalgebras} if the following conditions hold:

(1) Each graded component $A_n$ is a finite dimensional superalgebra with unit $1_n$, and $A_0\cong\bk$.

(2) The restriction $\mu_{m,n}:A_m\otimes A_n\rightarrow A_{m+n}$ of multiplication $\mu$ is a homomorphism of superalgebras for all $m,n\geq0$, sending $1_m\ot1_n$ to $1_{m+n}$.

(3) For all $m,n\geq0$, $\mu_{m,n}$ induces a two-sided projective $A_m\otimes A_n$-module structure on  $A_{m+n}$ defined by
\[(a\ot b).c:=\mu_{m,n}(a\ot b)c,\,c.(a\ot b):=c\mu_{m,n}(a\ot b)\]
for any $a\in A_m,b\in A_n, c\in A_{m+n}$.
\end{defn}

For a tower $A=\bigoplus_{n\geq0}A_n$ of superalgebras, we define the categories
\[A\mb{-mod}:=\bigoplus_{n\geq0}A_n\mb{-mod},\,
A\mb{-pmod}:=\bigoplus_{n\geq0}A_n\mb{-pmod}.\]
and the corresponding Grothendieck groups
\[\mG(A):=\bigoplus_{n\geq0}K_0(A_n\mb{-mod}),\,
\mK(A):=\bigoplus_{n\geq0}K_0(A_n\mb{-pmod}).\]

Both the Cartan map (\ref{car}) and the pair (\ref{pair}) can be linearly extended to one for towers of superalgebras. Define
\beq\label{car1}\chi:=\bigoplus_{n\geq0}\chi_n:\mK(A)\rw\mG(A),\eeq
where $\chi_n$ is the Cartan map (\ref{car}) of $A_n\,(n\in\bN)$. And
\[\lan \cdot,\cdot \ran: \mK(A)\times\mG(A)\rw\bZ\]
by
\beq\label{pa1}
\lan [P],[M] \ran:=\begin{cases}
\mb{dim}_\bk\mb{Hom}_{A_n}(P,M),& P\in A_n\mb{-pmod},
M\in A_n\mb{-mod}\mb{ for some }$n$,\\
0,&\mb{otherwise}.
\end{cases}\eeq

For any $r\in\bN$, write $A_{n_1,\dots,n_r}:=A_{n_1}\ot\cdots \ot A_{n_r}$ and define
\[A\mb{-mod}^{\ot r}:=\bigoplus_{n_1,\dots,n_r\geq0}A_{n_1,\dots,n_r}\mb{-mod},\,
A\mb{-pmod}^{\ot r}:=\bigoplus_{n_1,\dots,n_r\geq0}A_{n_1,\dots,n_r}\mb{-pmod}.\]
Note that the formula (\ref{tensor}) can be further extended on $A\mb{-pmod}^{\ot r}\times A\mb{-mod}^{\ot r}$.

For any $\si\in\fs_r$, define $\tau_\si:A\mb{-mod}^{\ot r}\rw A\mb{-mod}^{\ot r}$ to be the functor twisting module structure by the following superalgebra isomorphism from $A_{n_1,\dots,n_r}$
to $A_{n_{\si^{-1}(1)},\dots,n_{\si^{-1}(r)}}$:
\beq\label{perm}a_1\ot\cdots\ot a_r\mapsto(-1)^{d_\si}
a_{\si^{-1}(1)}\ot\cdots\ot a_{\si^{-1}(r)},\eeq
where $d_\si:=\sum_{1\leq i<j\leq r\atop\si(i)>\si(j)}|a_i||a_j|$. Write $\tau_{ij}:=\tau_{s_{ij}}$ for short. For any $M\in A_{n_{\si^{-1}(1)},\cdots,n_{\si^{-1}(r)}}\mb{-mod}$, we denote its twisted $A_{n_1,\dots,n_r}$-supermodule structure by $^{\tau_\si}M$.

Now we have the following result.
\begin{lem}\label{twist}
Given $M_i\in A_{n_i}\mb{-mod}\,(1\leq i\leq r)$ and $\si\in\fs_n$, there exists the following (even) $A_{n_1,\dots,n_r}$-supermodule isomorphism
\[\begin{split}
\Psi:\,&M_1\ot\cdots\ot M_r\rw{^{\tau_\si}(M_{\si^{-1}(1)}\ot\cdots\ot M_{\si^{-1}(r)})},\\
&m_1\ot\cdots\ot m_r\mapsto (-1)^{d'_\si}m_{\si^{-1}(1)}\ot\cdots\ot m_{\si^{-1}(r)},
\end{split}\]
where $d'_\si:=\sum_{1\leq i<j\leq r\atop\si(i)>\si(j)}|m_i||m_j|$.
\end{lem}
\begin{proof}On one hand,
\[\begin{split}
\Psi&((a_1\ot\cdots\ot a_r).(m_1\ot\cdots\ot m_r))=
(-1)^{\sum_{1\leq i<j\leq r}|a_j||m_i|}\Psi(a_1m_1\ot\cdots\ot a_rm_r)\\
&=(-1)^{\sum_{1\leq i<j\leq r}|a_j||m_i|+
\sum_{1\leq i<j\leq r\atop\si(i)>\si(j)}(|a_i|+|m_i|)(|a_j|+|m_j|)}
a_{\si^{-1}(1)}m_{\si^{-1}(1)}\ot\cdots\ot a_{\si^{-1}(r)}m_{\si^{-1}(r)}.
\end{split}\]
On the other hand,
\[\begin{split}
(a_1&\ot\cdots\ot a_r)._{\tau_\si}\Psi(m_1\ot\cdots\ot m_r)\\
&=(-1)^{\sum_{1\leq i<j\leq r\atop\si(i)>\si(j)}(|a_i||a_j|+|m_i||m_j|)}
(a_{\si^{-1}(1)}\ot\cdots\ot a_{\si^{-1}(r)}).(m_{\si^{-1}(1)}\ot\cdots\ot m_{\si^{-1}(r)}),\\
&=(-1)^{\sum_{1\leq i<j\leq r\atop\si(i)>\si(j)}(|a_i||a_j|+|m_i||m_j|)+\sum_{1\leq i<j\leq r}|a_{\si^{-1}(j)}||m_{\si^{-1}(i)}|}
a_{\si^{-1}(1)}m_{\si^{-1}(1)}\ot\cdots\ot a_{\si^{-1}(r)}m_{\si^{-1}(r)}.
\end{split}
\]
They are equal since
\[\begin{split}
\sum_{1\leq i<j\leq r}&|a_j||m_i|+\sum_{1\leq i<j\leq r\atop\si(i)>\si(j)}(|a_i||m_j|+|a_j||m_i|)\\
&=\sum_{1\leq i<j\leq r\atop\si(i)<\si(j)}|a_j||m_i|+\sum_{1\leq i<j\leq r\atop\si(i)>\si(j)}|a_i||m_j|
=\sum_{1\leq i<j\leq r}|a_{\si^{-1}(j)}||m_{\si^{-1}(i)}|
\end{split}\]
in $\bZ_2$.
\end{proof}


Given an even homomorphism $f:B\rw A$ of superalgebras, the usual induction and restriction functors are defined by
\[\begin{array}{l}
\mb{Ind}^A_B: B\mb{-mod}\rw A\mb{-mod},\,\mb{Ind}^A_B M:=A^f\ot_B N,\,N\in B\mb{-mod},\\
\mb{Res}^A_B: A\mb{-mod}\rw B\mb{-mod},\,\mb{Res}^A_B M:=\mb{Hom}_A(A^f,M)\cong {^f}A\ot_A M,\,M\in A\mb{-mod},
\end{array}\]
where the left $B$-action on $\mb{Hom}_A(A^f,M)$ is defined by
$(bg)(a)=(-1)^{|b|(|a|+|g|)}g(af(b)),\,a\in A,b\in B,g\in\mb{Hom}_A(A,M)$, and the above isomorphism is given by $g\mapsto1_A\ot g(1_A)$.

Now if a tower $A$ of superalgebras is fixed, we abbreviate $\mb{Ind}_{A_m\ot A_n}^{A_{m+n}}$ as
$\mb{Ind}_{m,n}^{m+n}$ and $\mb{Res}_{A_m\ot A_n}^{A_{m+n}}$ as $\mb{Res}_{m,n}^{m+n}$, which base on the multiplication
$\mu_{m,n}:A_m\otimes A_n\rightarrow A_{m+n}$. Define \[\mb{Ind}:=\bigoplus_{m,n\geq0}\mb{Ind}_{m,n}^{m+n},\,\mb{Res}:=\bigoplus_{m,n\geq0}\mb{Res}_{m,n}^{m+n}.\]

We are interested in those pair $(\mK(A),\mG(A))$, which forms a dual pair of graded Hopf algebras via such induction and restriction.
For such duality, we also need the notion of Frobenius superalgebras (see \cite[\S 6]{RS}).
\begin{defn}
A finite dimensional superalgebra $A$ is called a \textit{Frobenius superalgebra} if one of the following three equivalent conditions holds:

(a) There is an even left $A$-supermodule isomorphism $\rho:A\rw A^*$.

(b) There exists a nondegenerate invariant even $\bk$-bilinear form
$(\cdot,\cdot):A\times A\rw\bk$ such that $(a,b)=0$ if $a\in A_{\bar{0}},b\in A_{\bar{1}}$ or $a\in A_{\bar{1}},b\in A_{\bar{0}}$, and also
$(ab,c)=(a,bc),\,a,b,c\in A$.

(c) There exists an even $\bk$-linear map
$tr:A\rw\bk$, called the \textit{trace map}, such that $\mb{ker }tr$ contains no non-zero left ideals of $A$.

The relationship between these three conditions is as follows:
\[\rho(b)(a)=(-1)^{|a||b|}(a,b),\,(a,b)=tr(ab),\,a,b\in A.\]
There exists an even automorphism $\varphi$ of $A$ satisfying $(a,b)=(-1)^{|a||b|}(\varphi(b),a),\,a,b\in A$. This automorphism is called the \textit{Nakayama automorphism} of $A$.
\end{defn}

\begin{lem}
Let $A_i\,(i=1,2)$ be two Frobenius superalgebras with trace maps $tr_i\,(i=1,2)$ and Nakayama automorphisms $\varphi_i\,(i=1,2)$ respectively. Then $A_1\ot A_2$ is also a Frobenius superalgebra with trace map $tr_1\ot tr_2$ and Nakayama automorphism $\varphi_1\ot \varphi_2$.
\end{lem}

For a tower $A$ of superalgebras such that each $A_n$ is a Frobenius superalgebra with Nakayama automorphism $\varphi_n$. We abuse the notation $\varphi_n$ to be the automorphism on $A_n$-mod and $A_n$-pmod twisting modules by $\varphi_n$ of $A_n$. Then $\varphi:=\bigoplus_{n\geq0}\varphi_n$ is an automorphism of $A$-mod and $A$-pmod, thus induces an automorphism on $\mK(A)$ and $\mG(A)$.

\begin{prop}\cite[Prop. 6.7]{RS}\label{dual}
Under the above assumption, the induction is conjugate right adjoint to restriction with conjugation $\varphi$, i.e. for any $m,n\geq0$, $M\in A_m\mb{-mod},N\in A_n\mb{-mod},L\in A_{m+n}\mb{-mod}$, the following functorial isomorphism holds:
\[\mb{Hom}_{A_{m+n}}\lb L,\mb{Ind}_{m,n}^{m+n}(M\ot N)\rb\cong
\mb{Hom}_{A_m\ot A_n}\lb{_\varphi\mb{Res}_{m,n}^{m+n}}(L),M\ot N\rb,\]
where $_\varphi\mb{Res}_{m,n}^{m+n}:=(\varphi_m\ot\varphi_n)\circ\mb{Res}_{m,n}^{m+n}\circ\varphi^{-1}_{m+n}$.
\end{prop}



%

\section{Representation theory of 0-Hecke-Clifford algebras}

In this section we further study the representation theory of 0-Hecke-Clifford algebras in detail, referring to the discussion in \cite[\S 5]{BHT}. The \textit{0-Hecke-Clifford algebra} $HCl_n(0)$ of type A is an algebra generated by $T_i,\,1\leq i\leq n-1;\,c_j,\,1\leq j\leq n$, where $T_i$'s generate the \textit{0-Hecke algebra} $H_n(0)$ with relations
\[\begin{array}{l}
T_i^2=-T_i,\,1\leq i\leq n-1,\\
T_iT_j=T_jT_i,\,|i-j|>1,\\
T_iT_{i+1}T_i=T_{i+1}T_iT_{i+1},\,1\leq i\leq n-2,
\end{array}\]
and $c_j$'s generate the \textit{Clifford algebra} $Cl_n$ with relations
\[c_i^2=-1,\,1\leq i\leq n;\quad c_ic_j=-c_jc_i,\,i\neq j,\]
while the two parts satisfy the following cross-relations
\[\begin{array}{l}
T_ic_j=c_jT_i,\,j\neq i,i+1,\\
T_ic_i=c_{i+1}T_i,\,1\leq i\leq n-1,\\
(T_i+1)c_{i+1}=c_i(T_i+1),\,1\leq i\leq n-1.
\end{array}\]
Let $\mbox{deg}(T_i)=\bar{0},\,\mbox{deg}(c_j)=\bar{1}$, then $HCl_n(0)$ becomes a $\bZ_2 $-graded superalgebra, with a linear basis $\{c_DT_w:D\subseteq[n],w\in\fs_n\}$, where $c_D:=c_{i_1}\cdots c_{i_r}$ for $D=\{i_1<\cdots<i_r\}$ and $T_w:=T_{j_1}\cdots T_{j_s}$ for any reduced expression $w=s_{j_1}\cdots s_{j_s}$. Throughout this paper, we only consider 0-Hecke(-Clifford) algebras, thus write $\mH_n:=H_n(0),\, \mH\mC_n:=HCl_n(0)$ for short.

First recall the main result about representation theory of 0-Hecke-Clifford algebras discussed in \cite[\S 5]{BHT}. Note that the 0-Hecke algebra $\mH_n$ of type A is a basic algebra whose simple modules are one-dimensional, and a complete set of non-isomorphic simple modules are indexed by the compositions of $n$. In fact, given $\al\vDash n$, the simple module $S_\al:=\bk \eta_\al$ is defined as follows.
\beq\label{hec}T_i.\eta_\alpha=\begin{cases}
-\eta_\alpha,& i\in D(\alpha),\\
0,&\mbox{otherwise}.
\end{cases}\eeq
We denote by $P_\al$ the projective cover of $S_\al$. It has a linear basis $\{u_w:w\in\fd_\al\}$ with the module structure given as follows \cite[\S 5.3]{KT}:
\beq\label{act}
T_i.u_w=\begin{cases}
-u_w, &i\in D(w^{-1}),\\
u_{s_iw}, &i\notin D(w^{-1}), s_iw\in\fd_\al,\\
0,&\mbox{otherwise}.
\end{cases}
\eeq
Moreover, for the tower $\mH:=\bigoplus_{n\geq0}\mH_n$, Krob and Thibon defined the following Frobenius isomorphism and its adjoint in \cite{KT}:
\beq\label{fro}\mbox{Ch}:\mathcal{G}(\mH)\rightarrow\mbox{QSym},\, [S_\alpha]\mapsto F_\alpha,\,
\mbox{Ch}^*:\mbox{NSym}\rightarrow\mathcal{K}(\mH),\,R_\alpha\mapsto[P_\alpha].\eeq
Both are graded Hopf algebra isomorphisms satisfying
\[\langle F,\mbox{Ch}([M])\rangle=\langle\mbox{Ch}^*(F),[M]\rangle,\,
F\in\mbox{NSym},[M]\in\mathcal{G}(\mH).\]
Due to Gessel, we know that the Cartan map $\chi$ of $\mH$ satisfies the following commutative diagram \cite[Prop. 5.9]{KT}:
\[\xymatrix@=2em{
\mK(\mH)\ar@{->}[r]^-\chi&\mG(\mH)
\ar@{->}[d]^-{\mb{\scriptsize Ch}}\\
\mb{NSym}\ar@{->}[r]^-\pi\ar@{->}[u]^-{\mb{\scriptsize Ch}^*}
&\mb{QSym}}\]
That is, $\chi([P_\al])=\sum_\be\lan R_\be,r_\al\ran[S_\be]$, where $r_\al:=\pi(R_\al)\in\La$ is the \textit{ribbon Schur function} of shape $\al$.

\subsection{Irreducible supermodules of 0-Hecke-Clifford algebras}
A complete set of non-isomorphic simple supermodules of $\mH\mC_n$ are parameterized by peak sets and thus denoted by $\{HClS_P\}$. They can be obtained from the induced modules
\[\tilde{S}_\alpha:=\mbox{Ind}_{\mH_n}^{\mH\mC_n}S_\alpha,\,\al\vDash n.\]
Slightly modifying \cite[Theorem 5.4]{BHT}, we have
\begin{theorem}\label{end}
Let $\al\vDash n$, $V:=V(\al)\subseteq[n]$ and $Cl_V$ be the subalgebra of $Cl_n$ generated by $\{c_v\}_{v\in V}$. For any homogeneous $c\in Cl_V$, define $f_c\in\mb{End}_\bk(\tilde{S}_\al)$ by
\[f_c(c_D\eta_\al)=(-1)^{|c||D|}c_Dc\eta_\al,\,D\subseteq[n].\]
If we linearly extend it to all $c\in Cl_V$, then the map $c\mapsto f_c$ is an even isomorphism from $Cl_V$ to $\mb{End}_{\mH\mC_n}(\tilde{S}_\al)$.
\end{theorem}

By \cite[Th. 5.5, Cor. 5.6]{BHT} we know that $\tilde{S}_\al\cong\tilde{S}_\be$ (even isomorphism) if and only if $P(\al)=P(\be)$ and
\beq\label{dec}\tilde{S}_\al\cong {HClS_{P(\al)}}^{\oplus 2^{l_\al}},\,l_\al:=\left\lfloor\tfrac{|P(\al)|+1}{2}\right\rfloor.\eeq
In fact, we know that for $V(\al)=\{n_1,\dots,n_s\}$, if define the mutually orthogonal even idempotents
\beq\label{ide}e^\ve_\al:=\dfrac{1}{2^{l_\al}}(1+\ve_1\sqrt{-1}c_{n_1}c_{n_2})
(1+\ve_2\sqrt{-1}c_{n_3}c_{n_4})\cdots(1+\ve_{l_\al}\sqrt{-1}c_{n_{2l_\al-1}}
c_{n_{2l_\al}}),\eeq
where $\ve=(\ve_1,\dots,\ve_{l_\al})\in\{\pm1\}^{\times{l_\al}}$, then $Cl_n e^\ve_\al\eta_\al\,(\ve\in\{\pm1\}^{\times l_\al})$ provide all the simple components of $\tilde{S}_\al$.

\begin{prop}\label{type}
For any peak set $P$, the irreducible supermodule $HClS_P$ is of type M when $|P|$ is odd, and of type Q when $|P|$ is even.
\end{prop}
\begin{proof}
By Theorem \ref{end}, we know that $\mb{End}_{\mH\mC_n}(\tilde{S}_\al)\cong Cl_V$ for any $\al\vDash n$ and $V:=V(\al)\subseteq[n]$. Since $|V(\al)|=|P(\al)|+1$, by Schur's Lemma we only need to prove that there exists an odd automorphism of $\tilde{S}_\al$ stabilizing any simple component of it if and only if $|V|$ is odd.

First for any even idempotent $e^\ve_\al$ in (\ref{ide}), it is easy to check that
\[e^\ve_\al c_i=\begin{cases}
c_i e^\ve_\al,& i\notin\{n_1,\dots,n_{2l_\al}\},\\
c_i e^{\ve^{(i)}}_\al,& i\in\{n_1,\dots,n_{2l_\al}\},
\end{cases}\]
where $V=\{n_1,\dots,n_s\}$ and $\ve^{(i)}$ is obtained from $\ve$ by only changing the sign $\ve_j$ if $i\in\{n_{2j-1},n_{2j}\}\subseteq V$. It means that the map $f_{c_E}\,(E\subseteq V)$ defined in Theorem \ref{end} stabilizes any simple component of  $\tilde{S}_\al$ if and only if $|E\cap\{n_{2j-1},n_{2j}\}|$ is even for all $j=1,\dots,l_\al$.
Hence, if we need $f_{c_E}$ to be odd and satisfies the stability condition, then $V$ must be odd too.
\end{proof}

Denote
\[\mathcal{G}:=\mathcal{G}(\mH),\,\mathcal{K}:=\mathcal{K}(\mH),
\,\widetilde{\mathcal{G}}:=\mathcal{G}(\mH\mC),
\,\widetilde{\mathcal{K}}:=\mathcal{K}(\mH\mC).\]
In \cite{BHT} Bergeron et al. nicely define the following Frobenius isomorphism for the tower of 0-Hecke-Clifford algebras:
\[\widetilde{\mbox{Ch}}:\widetilde{\mathcal{G}}\rightarrow\mbox{Peak}^*,\,[\tilde{S}_\al]\mapsto K_{P(\alpha)},\]
which satisfies the following commutative diagrams (as a categorification of the descent-to-peak map):
\beq\label{dp}\xymatrix@=2em{
\mathcal{G}\ar@{->}[r]^-{\mbox{\scriptsize Ind}_{\mH}^{\mH\mC}}
\ar@{->}[d]^-{\mbox{\scriptsize Ch}}&\widetilde{\mathcal{G}}
\ar@{->}[d]^-{\widetilde{\mbox{\scriptsize Ch}}}\\
\mbox{QSym}\ar@{->}[r]^-{\vartheta}
&\mbox{Peak}^*},\quad
\xymatrix@=2em{[S_\alpha]\ar@{|->}[r]
\ar@{|->}[d]
&[\tilde{S}_\alpha]\ar@{|->}[d]\\
F_\alpha\ar@{|->}[r]&K_{P(\alpha)}}\eeq
\beq\label{emb}\xymatrix@=2em{
\widetilde{\mathcal{G}}\ar@{->}[r]^-{\mbox{\scriptsize Res}_{\mH}^{\mH\mC}}
\ar@{->}[d]^-{\widetilde{\mbox{\scriptsize Ch}}}&\mathcal{G}
\ar@{->}[d]^-{\mbox{\scriptsize Ch}}\\
\mbox{Peak}^*\ar@{->}[r]
&\mbox{QSym}},\quad
\xymatrix@=2em{[\tilde{S}_\alpha]\ar@{|->}[r]
\ar@{|->}[d]
&[\mbox{Res}_{\mH}^{\mH\mC}\tilde{S}_\alpha]\ar@{|->}[d]\\
K_{P(\alpha)}\ar@{|->}[r]&
2^{|P(\alpha)|+1}\sum\limits_{\beta\vDash n\atop P(\alpha)\subseteq D(\beta)\triangle(D(\beta)+1)}F_\beta}\eeq
where $\mbox{Ind}_{\mH}^{\mH\mC}:=\bigoplus_{n\geq0}\mbox{Ind}_{\mH_n}^{\mH\mC_n}$
and $\mbox{Res}_{\mH}^{\mH\mC}:=\bigoplus_{n\geq0}\mbox{Res}_{\mH_n}^{\mH\mC_n}$.

\subsection{Projective supermodules of 0-Hecke-Clifford algebras}
In this subsection we consider the category of finitely generated projective supermodules of 0-Hecke-Clifford algebras, which is lack of  discussion in \cite{BHT}. So far it is not so easy to construct the indecomposable one directly, we similarly consider the induction from
the projective modules of 0-Hecke algebras. Define
\[\tilde{P}_\alpha:=\mbox{Ind}_{\mH_n}^{\mH\mC_n}P_\alpha,\,\al\vDash n,\]
which becomes a projective $\mH\mC_n$-supermodule.


Since $\mH\mC_n$ is a free right $\mH_n$-module of rank $2^n$,
we can write a basis for $\tilde{P}_\al$ as $\{c_Du_w:D\subseteq[n],w\in\fd_\al\}$ for short.
Via the defining relation of $\mH\mC_n$, we get the following commutation relations:
\beq\label{com}T_ic_D=\begin{cases}
c_DT_i, &i,i+1\notin D,\\
c_{(D\backslash\{i\})\cup\{i+1\}}T_i, &i\in D,\,i+1\notin D,\\
c_{(D\backslash\{i+1\})\cup\{i\}}(T_i+1)-c_D, &i\notin D,\,i+1\in D,\\
-c_D(T_i+1)+c_{D\backslash\{i,i+1\}},&i,i+1\in D,
\end{cases}\eeq
for any $D\subseteq[n],\,i=1,\dots,n-1$.
Combining these relations and the module action (\ref{act}) of $P_\al$, one can describe explicitly the module structure of  $\tilde{P}_\al$ for any $\al\vDash n$.

For example, the module structure of $\tilde{P}_{12}$ can
be depicted explicitly by the following graphs (separated into two $\bZ_2$-graded components):
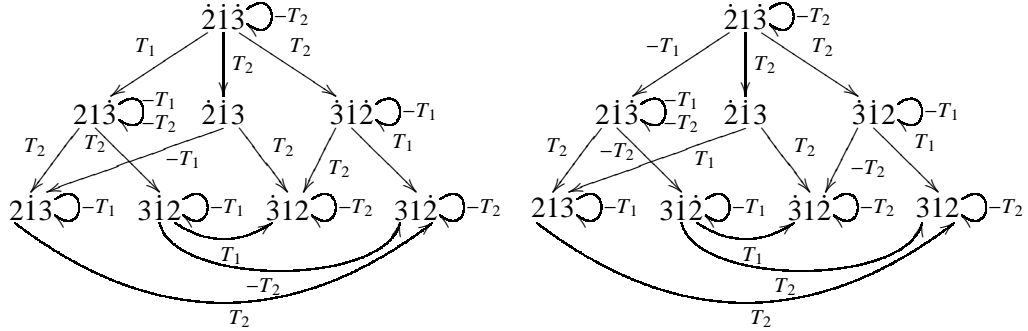
\[\xy 0;/r.25pc/:
{\ar_{T_1}(-2,22)*{};(-14,14)*{}};
{\ar^{T_2}(2,22)*{};(14,14)*{}};
{\ar^{T_2}(0,22)*{};(0,14)*{}};
{(0,10)*{};(-22,2)*{}**\crv{}?(1)*\dir{>}+(17,4)*{\scriptstyle -T_1}};
{\ar_{T_2}(-18,10)*{};(-24,2)*{}};
{(-16,10)*{};(-8,2)*{}**\crv{}?(1)*\dir{>}+(-8,6)*{\scriptstyle T_2}};
{\ar^{T_1}(16,10)*{};(24,2)*{}};
{\ar^{T_2}(14,10)*{};(10,2)*{}};
{\ar^{T_2}(2,10)*{};(8,2)*{}};
{(-13,13)*{};(-13,11)*{}**\crv{(-9,16) & (-9,8)}?(1)*\dir{>}+(5,1)*{-T_1\atop-T_2}};
{(19,13)*{};(19,11)*{}**\crv{(23,16) & (23,8)}?(1)*\dir{>}+(5.5,1)*{\scriptstyle-T_1}};
{(3,25)*{};(3,23)*{}**\crv{(7,28) & (7,20)}?(1)*\dir{>}+(5.5,1)*{\scriptstyle-T_2}};
{(27,1)*{};(27,-1)*{}**\crv{(31,4) & (31,-4)}?(1)*\dir{>}+(5.5,1)*{\scriptstyle-T_2}};
{(11,1)*{};(11,-1)*{}**\crv{(15,4) & (15,-4)}?(1)*\dir{>}+(5.5,1)*{\scriptstyle-T_2}};
{(-5,1)*{};(-5,-1)*{}**\crv{(-1,4) & (-1,-4)}?(1)*\dir{>}+(5.5,1)*{\scriptstyle-T_1}};
{(-21,1)*{};(-21,-1)*{}**\crv{(-17,4) & (-17,-4)}?(1)*\dir{>}+(5.5,1)*{\scriptstyle-T_1}};
{(-6,-2)*{};(6,-2)*{}**\crv{(0,-6)}
?(1)*\dir{>}+(-5,-4)*{\scriptstyle T_1}};
{(-8,-2)*{};(22,-2)*{}**\crv{(-6,-10)&(20,-10)}
?(1)*\dir{>}+(-17,-8)*{\scriptstyle -T_2}};
{(-26,-2)*{};(26,-2)*{}**\crv{(0,-22)}
?(1)*\dir{>}+(-24,-12)*{\scriptstyle T_2}};
(0,24)*{\dot{2}\dot{1}\dot{3}};(-16,12)*{21\dot{3}};(0,12)*{\dot{2}\dot{1}3};
(16,12)*{\dot{3}\dot{1}\dot{2}};(-24,0)*{2\dot{1}3};
(-8,0)*{3\dot{1}2};(8,0)*{\dot{3}12};(24,0)*{31\dot{2}};
\endxy\quad
\xy 0;/r.25pc/:
{\ar_{-T_1}(-2,22)*{};(-14,14)*{}};
{\ar^{T_2}(2,22)*{};(14,14)*{}};
{\ar^{T_2}(0,22)*{};(0,14)*{}};
{(0,10)*{};(-22,2)*{}**\crv{}?(1)*\dir{>}+(17,4)*{\scriptstyle T_1}};
{\ar_{T_2}(-18,10)*{};(-24,2)*{}};
{(-16,10)*{};(-8,2)*{}**\crv{}?(1)*\dir{>}+(-8,5)*{\scriptstyle -T_2}};
{\ar^{T_1}(16,10)*{};(24,2)*{}};
{\ar^{-T_2}(14,10)*{};(10,2)*{}};
{\ar^{T_2}(2,10)*{};(8,2)*{}};
{(-13,13)*{};(-13,11)*{}**\crv{(-9,16) & (-9,8)}?(1)*\dir{>}+(5,1)*{-T_1\atop-T_2}};
{(19,13)*{};(19,11)*{}**\crv{(23,16) & (23,8)}?(1)*\dir{>}+(5.5,1)*{\scriptstyle-T_1}};
{(3,25)*{};(3,23)*{}**\crv{(7,28) & (7,20)}?(1)*\dir{>}+(5.5,1)*{\scriptstyle-T_2}};
{(27,1)*{};(27,-1)*{}**\crv{(31,4) & (31,-4)}?(1)*\dir{>}+(5.5,1)*{\scriptstyle-T_2}};
{(11,1)*{};(11,-1)*{}**\crv{(15,4) & (15,-4)}?(1)*\dir{>}+(5.5,1)*{\scriptstyle-T_2}};
{(-5,1)*{};(-5,-1)*{}**\crv{(-1,4) & (-1,-4)}?(1)*\dir{>}+(5.5,1)*{\scriptstyle-T_1}};
{(-21,1)*{};(-21,-1)*{}**\crv{(-17,4) & (-17,-4)}?(1)*\dir{>}+(5.5,1)*{\scriptstyle-T_1}};
{(-6,-2)*{};(6,-2)*{}**\crv{(0,-6)}
?(1)*\dir{>}+(-5,-4)*{\scriptstyle T_1}};
{(-8,-2)*{};(22,-2)*{}**\crv{(-6,-10)&(20,-10)}
?(1)*\dir{>}+(-17,-8)*{\scriptstyle T_2}};
{(-26,-2)*{};(26,-2)*{}**\crv{(0,-22)}
?(1)*\dir{>}+(-24,-12)*{\scriptstyle T_2}};
(0,24)*{\dot{2}1\dot{3}};(-16,12)*{2\dot{1}\dot{3}};(0,12)*{\dot{2}\dot{1}3};
(16,12)*{\dot{3}\dot{1}2};(-24,0)*{213};
(-8,0)*{3\dot{1}\dot{2}};(8,0)*{\dot{3}1\dot{2}};(24,0)*{312};
\endxy
\]
where we use $w\in\fd_\al$ to represent $u_w$ and put dots on the heads of those numbers to represent basis elements $c_Du_w$, e.g. we abbreviate $c_{\{1,3\}}u_{213}$ as $\dot{2}1\dot{3} $.

By relation (\ref{com}), one easily gets the following result.
\begin{lem}\label{ord}
For any $w\in\fs_n,D\subseteq[n]$,
\[T_wc_D=(-1)^{l_{w,D}}c_{w(D)}T_w+\sum_{E\subseteq D\atop v<w}a_{E,v}c_ET_v\]
for some integers $a_{E,v}$, where $l_{w,D}:=|\{i,j\in D:i<j,w(i)>w(j)\}|$, $w(D):=\{w(i):i\in D\}$ and $<$ stands for the Bruhat order of $\fs_n$.
\end{lem}

\subsection{Dual Hopf algebras arising from 0-Hecke-Clifford algebras}
In this subsection, we check a series of axioms for the tower of 0-Hecke-Clifford algebras in order to show that
$(\widetilde{\mK},\,\widetilde{\mG})$ forms a dual pair of graded Hopf algebras. One of the key steps is to define a proper Hopf
pairing.

First the Mackey property of 0-Hecke-Clifford algebras is easy to proved by mimicing the nice approach in \cite[\S 2-h.]{BK1} for affine Hecke-Clifford algebras. That guarantees both of the Grothendieck groups $\widetilde{\mK}$ and $\widetilde{\mG}$ to be Hopf algebras via the induction and the restriction. One can also refer to \cite[Theorem 2.7]{DJ}, \cite[Prop. 4.3]{SY} for the case of Hecke algebras.

For the sake of completeness, we sketch the proof steps as follows.
For any $\al=(\al_1,\dots,\al_r)\vDash n$, define $\mH\mC_\al:=\mH\mC_{\al_1}\ot\cdots\ot \mH\mC_{\al_r}$, embedding as a parabolic subalgebra of $\mH\mC_n$. Similarly define the subalgebra $\mH_\al$ of $\mH_n$.

Given $\al,\be\vDash n$, let $\fs_\al:=\lan s_i:i\notin D(\al)\ran$ be the Young subgroup of $\fs_n$, $\fr_\al$ denote the set of minimal length left $\fs_\al$-coset representatives in $\fs_n$, and $\fr_\al^{-1}$ for the one corresponding to right $\fs_\al$-coset. Then $\fr_{\al,\be}:=\fr^{-1}_\al\cap\fr_\be$ is the
set of minimal length $(\fs_\al,\fs_\be)$-double coset representatives in $\fs_n$. For any $x\in\fr_{\al,\be}$, $\fs_\al\cap x\fs_\be x^{-1}$ and $x^{-1}\fs_\al x\cap\fs_\be$ are Young subgroups of $\fs_n$, thus we can let $\al\cap x\be$ and $x^{-1}\al\cap \be$ to be the two compositions of $n$ such that
\[\fs_\al\cap x\fs_\be x^{-1}=\fs_{\al\cap x\be},\,
x^{-1}\fs_\al x\cap\fs_\be=\fs_{x^{-1}\al\cap\be}.\]
Now it needs several technical lemmas as follows.
\begin{lem}
For any $x\in\fr_{\al,\be}$, the subspace $\mH\mC_\al T_x\mH_\be$ of $\mH\mC_n$ has basis $\{C_DT_w:D\subseteq[n],w\in\fs_\al x\fs_\be\}$. Moreover,
\[\mH\mC_n=\bigoplus_{x\in\fr_{\al,\be}}\mH\mC_\al T_x\mH_\be.\]
\end{lem}

Fix some total order $\prec$ refining the Bruhat order $<$ on $\fr_{\al,\be}$. For $x\in\fr_{\al,\be}$, let
\[\mB_{\preceq x}:=\bigoplus_{y\in\fr_{\al,\be},y\preceq x}\mH\mC_\al T_y \mH_\be,\,\mB_{\prec x}:=\bigoplus_{y\in\fr_{\al,\be},y\prec x}\mH\mC_\al T_y\mH_\be,\]
and $\mB_x:=\mB_{\preceq x}/\mB_{\prec x}$. By Lemma \ref{ord}, we know that $\mB_{\preceq x}$ (resp. $\mB_{\prec x}$) is invariant under right multiplication by $Cl_n$. Hence, $\{\mB_{\preceq x}\}_{x\in\fr_{\al,\be}}$ is an $(\mH\mC_\al,\mH\mC_\be)$-bimodule filtration of $\mH\mC_n$.

\begin{lem}
For any $x\in\fr_{\al,\be}$, there exists an algebra isomorphism
\[\phi=\phi_x:\mH\mC_{\al\cap x\be}\rw \mH\mC_{x^{-1}\al\cap\be}\]
with $\phi(T_w)=T_{x^{-1}wx}$, $\phi(c_i)=c_{x^{-1}(i)}$ for $w\in\fs_{\al\cap x\be},\,1\geq i\geq n$.
\end{lem}

\begin{lem}
View $\mH\mC_\al$ as an $(\mH\mC_\al,\mH\mC_{\al\cap x\be})$-bimodule and $\mH\mC_\be$ as an $(\mH\mC_{x^{-1}\al\cap\be},\mH\mC_\be)$-bimodule. Then $^\phi\mH\mC_\be$ is an $(\mH\mC_{\al\cap x\be},\mH\mC_\be)$-bimodule and
\[\Phi:\mB_x\rw\mH\mC_\al\ot_{\mH\mC_{\al\cap x\be}}{^\phi\mH\mC_\be},\,
uT_xv+\mB_{\prec x}\mapsto u\ot v,\, u\in\mH\mC_\al,v\in\mH\mC_\be\]
is an isomorphism of $(\mH\mC_\al,\mH\mC_\be)$-bimodule.
\end{lem}

\begin{theorem}[Mackey Theorem]\label{mac}
Let $\al,\be\vDash n$ and $M$ be an $\mH\mC_\be$-module. Then the $\mH\mC_\al$-module $\mb{Res}_{\mH\mC_\al}^{\mH\mC_n}\mb{Ind}_{\mH\mC_\be}^{\mH\mC_n}M$
admits an $\mH\mC_\al$-submodule filtration with subquotients isomorphic to $\mb{Ind}_{\mH\mC_{\al\cap x\be}}^{\mH\mC_\al}
\lb{^\phi\mb{Res}_{\mH\mC_{x^{-1}\al\cap\be}}^{\mH\mC_\be}M}\rb$, one for each $x\in\fr_{\al,\be}$.
\end{theorem}
\noindent
Especially when $\al,\be$ both have two parts, Theorem \ref{mac} implies that $\widetilde{\mG}$ has graded Hopf algebra structure under induction and restriction. The case for $\widetilde{\mK}$ is similar, since $\mB_x$ is projective as a left $\mH\mC_\al$-module.

In order to prove the Hopf duality between $\widetilde{\mK}$ and $\widetilde{\mG}$ by Prop. \ref{dual}, we continue to show that $\mH\mC_n$ is a Frobenius superalgebra. It is straightforward to check that
\begin{prop}
There exist two even algebra involutions $\varphi,\varphi'$ for the 0-Hecke-Clifford algebra $\mH\mC_n$, defined by
\beq\label{Nak}
\begin{array}{l}
\varphi(T_i)=T_{n-i}+c_{n-i}c_{n+1-i},\,i=1,\dots,n-1,\\
\varphi(c_j)=-c_{n+1-j},\,j=1,\dots,n.\\
\varphi'(T_i)=-(T_{n-i}+1),\,i=1,\dots,n-1,\\
\varphi'(c_j)=-c_{n+1-j},\,j=1,\dots,n.
\end{array}
\eeq
There exist two unsigned even algebra anti-involutions $\psi,\psi'$ of $\mH\mC_n$ defined by
\beq\label{anti1}
\begin{array}{l}
\psi(T_i)=T_i+c_ic_{i+1},\,i=1,\dots,n-1,\\
\psi(c_j)=-c_j,\,j=1,\dots,n.\\
\psi'(T_i)=-(T_i+1),\,i=1,\dots,n-1,\\
\psi'(c_j)=-c_j,\,j=1,\dots,n.
\end{array}
\eeq
\end{prop}

\begin{prop}\label{frob}
The 0-Hecke-Clifford algebra $\mH\mC_n$ is a Frobenius superalgebra with even trace map
\[tr_n:\mH\mC_n\rw\bk,\,tr_n(c_DT_w)=\de_{D,\emptyset}\de_{w,w_0},\,D
\subseteq[n],w\in\fs_n,\]
where $w_0$ is the longest element of $\fs_n$. Moreover, $\varphi$ is the corresponding \textit{Nakayama automorphism}.
\end{prop}
\begin{proof}
We only need to prove that $\mb{ker }tr_n$ contains no non-zero left ideals. Suppose $I$ is a non-zero left ideal of $\mH\mC_n$. We choose an element $b=\sum_{D\subseteq[n]\atop w\in\fs_n}b_{D,w}c_DT_w\in I\backslash\{0\}$ and let $\si$ be a maximal length element in the set $\{w\in\fs_n:b_{D',w}\neq0\mb{ for some }D'\subseteq[n]\}$. Then by Lemma \ref{ord} we have
\[\begin{split}
tr_n&(c_{w_0\si^{-1}(D')}T_{w_0\si^{-1}}b)
=b_{D',\si}tr_n(c_{w_0\si^{-1}(D')}T_{w_0\si^{-1}}c_{D'}T_\si)\\
&=(-)^{l_{w_0\si^{-1},D'}}b_{D',\si}tr_n(c_{w_0\si^{-1}(D')}^2T_{w_0})
=(-)^{l_{w_0\si^{-1},D'}+{|D'|+1\choose2}}b_{D',\si}\neq0.
\end{split}\]
Thus $I\nsubseteq\mb{ker }tr_n$. To show that $\varphi$ is the corresponding Nakayama automorphism, it suffices to show that

(1)\quad$tr_n(c_DT_wT_i)=tr_n((T_{n-i}+c_{n-i}c_{n+1-i})c_DT_w)$,

(2)\quad$tr_n(c_DT_wc_j)=(-1)^{|D|}tr_n(-c_{n+1-j}c_DT_w)$

\noindent
for all $i\in\{1,\dots,n-1\},j\in\{1,\dots,n\},D\subseteq[n]\mb{ and }w\in\fs_n$.

For (1) we break the proof into four cases as in \cite[Lemma 4.2]{SY}. When $w=w_0s_i$, $tr_n(c_DT_wT_i)=tr_n(c_DT_{w_0})=\de_{D,\emptyset}$. On the other hand,
\[\begin{split}
tr_n&((T_{n-i}+c_{n-i}c_{n+1-i})c_DT_w)
=tr_n(c_{s_{n-i}(D)}T_{n-i}T_w)+tr_n(c_{n-i}c_{n+1-i}c_DT_w)\\
&=tr_n(c_{s_{n-i}(D)}T_{w_0})=\de_{D,\emptyset},
\end{split}\]
where we use relation (\ref{com}) for the first equality and the identity $w_0s_i=s_{n-i}w_0$ for the second last one.

When $w=w_0$, $tr_n(c_DT_{w_0}T_i)=-tr_n(c_DT_{w_0})=-\de_{D,\emptyset}$. On the other hand, if $D\neq\emptyset,\{n-i,n+1-i\}$, then $tr_n((T_{n-i}+c_{n-i}c_{n+1-i})c_DT_{w_0})=0$ by relation (\ref{com}). Otherwise, for $D=\emptyset$,
\[tr_n((T_{n-i}+c_{n-i}c_{n+1-i})T_{w_0})
=-tr_n(T_{w_0})=-1.\]
For $D=\{n-i,n+1-i\}$,
\[\begin{split}
tr_n&((T_{n-i}+c_{n-i}c_{n+1-i})c_{n-i}c_{n+1-i}T_{w_0})
=tr_n((-c_{n-i}c_{n+1-i}(T_{n-i}+1)+1)T_{w_0})+tr_n(-T_{w_0})\\
&=tr_n(-c_{n-i}c_{n+1-i}(T_{n-i}+1)T_{w_0})=0.
\end{split}\]
The rest two cases when $\ell(w)\leq\ell(w_0)-2$ or $\ell(w)=\ell(w_0)-1$ but $w\neq w_0s_i$ are similar to check. For (2),
\[tr_n(c_DT_wc_j)=tr_n(c_Dc_{w(j)}T_w)=-\de_{D,\{w(j)\}}\de_{w,w_0}
=-\de_{D,\{n+1-j\}}\de_{w,w_0}=(-1)^{|D|}tr_n(-c_{n+1-j}c_DT_w).\]
\end{proof}


\begin{prop}
For any $\al\vDash n$,

(1) we have a supermodule isomorphism between
$\tilde{S}_{\al^*}$ and
$^\varphi(\tilde{S}_\al)$ of degree $\bar{n}$, sending $c_D\eta_{\al^*}$ to $(-1)^{n|D|}\varphi(c_D)c_{[n]}\eta_\al$. In particular, $^\varphi(HClS_{P(\al)})\cong HClS_{P(\al^*)}$.

(2) we have an even supermodule isomorphism between
$\tilde{S}_{\al^*}$ and
$^{\varphi'}(\tilde{S}_\al)$, sending $c_D\eta_{\al^*}$ to $\varphi'(c_D)\eta_\al$. In particular,
$^{\varphi'}(HClS_{P(\al)})\cong HClS_{P(\al^*)}$.

(3) we have an even supermodule isomorphism between
$\tilde{S}_\al$ and
$^\psi({\tilde{S}_\al}^*)$, sending $c_D\eta_\al$ to $c_D._\psi\zeta_\al$,
where $\zeta_\al$ is the dual of $\eta_\al$ with respect to the standard basis $\{c_D\eta_\al\}$ of $\tilde{S}_\al$. In particular, $^\psi({HClS_P}^*)\cong HClS_P$ for any peak set $P$ in $[n]$.

(4) we have a supermodule isomorphism between
$\tilde{S}_\al$ and
$^{\psi'}({\tilde{S}_\al}^*)$ of degree $\bar{n}$, sending $c_D\eta_\al$ to $(-1)^{n|D|}c_D._{\psi'}\xi_\al$,
where $\xi_\al$ is the dual of $c_{[n]}\eta_\al$ with respect to the standard basis $\{c_D\eta_\al\}$ of $\tilde{S}_\al$. In particular, $^{\psi'}({HClS_P}^*)\cong HClS_P$ for any peak set $P$ in $[n]$.
\end{prop}

\begin{proof}
(1) From relations (\ref{hec}) and (\ref{com}) we have
\[T_i._\varphi c_{[n]}\eta_\al=(T_{n-i}+c_{n-i}c_{n+1-i}).c_{[n]}\eta_\al
=-c_{[n]}(T_{n-i}+1).\eta_\al
=\begin{cases}
-c_{[n]}\eta_\al,&\mb{if }i\in D(\alpha^*),\\
0,&\mbox{otherwise}.
\end{cases}\]
Hence, there exists an $\mH_n$-module homomorphism from $S_{\al^*}$ to
$\mbox{Res}_{\mH_n}^{\mH\mC_n}\,{^\varphi(\tilde{S}_\al)}$, sending $\eta_{\al^*}$ to $c_{[n]}\eta_\al$. By the universal property of induction functors, we obtain an $\mH\mC_n$-supermodule homomorphism from $\tilde{S}_{\al^*}$ to $^\varphi(\tilde{S}_\al)$ sending $c_D\eta_{\al^*}$ to $(-1)^{n|D|}\varphi(c_D)c_{[n]}\eta_\al$. It is obviously surjective thus an isomorphism by dimension argument.

Meanwhile, for any $i\in[2,n-1]$,
\[i\in P(\al)\Leftrightarrow i\in D(\al),i-1\notin D(\al)
\Leftrightarrow n-i\notin D(\al^*),n+1-i\in D(\al^*)\Leftrightarrow n+1-i\in P(\al^*),\]
which means that $|P(\al)|=|P(\al^*)|$, thus
$HClS_{P(\al^*)}\cong{^\varphi(HClS_{P(\al)})}$ by (\ref{dec}).

(2) From relation (\ref{hec}) we have
\[T_i._{\varphi'}\eta_\al=-(T_{n-i}+1).\eta_\al
=\begin{cases}
-\eta_\al,& \mb{if }i\in D(\alpha^*),\\
0,&\mbox{otherwise}.
\end{cases}\]
Hence, there exists an $\mH_n$-module homomorphism from $S_{\al^*}$ to
$\mbox{Res}_{\mH_n}^{\mH\mC_n}\,{^{\varphi'}(\tilde{S}_\al)}$, sending $\eta_{\al^*}$ to $\eta_\al$. The rest of the proof is nearly the same as (1).

(3) By relation (\ref{com}) we have
\[\begin{split}
T_i._\psi\zeta_\al(c_D\eta_\al)&=\zeta_\al((T_i+c_ic_{i+1}).c_D\eta_\al)
=\de_{D,\emptyset}\zeta_\al(T_i.\eta_\al)\\
&=\begin{cases}
-1,& \mb{if }D=\emptyset,i\in D(\alpha),\\
0,& \mbox{otherwise}.
\end{cases}
\end{split}\]
That is, $T_i._\psi\zeta_\al=-\zeta_\al$ if $i\in D(\alpha)$ and 0 otherwise. Hence, there exists an $\mH_n$-module homomorphism from $S_\al$ to $\mbox{Res}_{\mH_n}^{\mH\mC_n}\,{^\psi({\tilde{S}_\al}^*)}$, sending $\eta_\al$ to $\zeta_\al$. Again by the universal property of induction functors, we obtain an $\mH\mC_n$-supermodule homomorphism from $\tilde{S}_\al$ to ${^\psi({\tilde{S}_\al}^*)}$, sending $c_D\eta_\al$ to $c_D._\psi\zeta_\al$. It is also obviously surjective thus an isomorphism by dimension argument.

(4) By relation (\ref{com}) we have
\[\begin{split}
T_i._{\psi'}\xi_\al(c_D\eta_\al)&=-\xi_\al((T_i+1).c_D\eta_\al)
=-\de_{D,[n]}\xi_\al((T_i+1).c_{[n]}\eta_\al)\\
&=-\de_{D,[n]}\xi_\al((-c_{[n]}T_i+c_{[n]\backslash\{i,i+1\}}).\eta_\al)
=\de_{D,[n]}\xi_\al(c_{[n]}T_i.\eta_\al)\\
&=\begin{cases}
-1,& \mb{if }D=[n],i\in D(\alpha),\\
0,&\mbox{otherwise}.
\end{cases}
\end{split}\]
That is, $T_i._{\psi'}\xi_\al=-\xi_\al$ if $i\in D(\alpha)$ and 0 otherwise. Hence, there exists an $\mH_n$-module homomorphism from $S_\al$ to $\mbox{Res}_{\mH_n}^{\mH\mC_n}\,{^{\psi'}({\tilde{S}_\al}^*)}$, sending $\eta_\al$ to $\xi_\al$. The rest of the proof is nearly the same as (3).
\end{proof}

For any $m,n\in\bN$, denote $\mH_{m,n}:=\mH_m\otimes \mH_n,\,\mH\mC_{m,n}:=\mH\mC_m\otimes \mH\mC_n$.
Let $\iota_n:\mH_n\rightarrow \mH\mC_n,\,\mu_{m,n}:\mH_{m,n}\rightarrow \mH_{m+n},\,\tilde{\mu}_{m,n}:\mH\mC_{m,n}\rightarrow \mH\mC_{m+n}$ be the natural embeddings.

\begin{prop}\label{flip}
For the tower of 0-Hecke-Clifford superalgebras, we have an isomorphism of functors
\[{_\varphi\mb{Res}}\cong\tau_{12}\circ\mb{Res}\]
on $\mH\mC$-mod (hence also on $\mH\mC$-pmod).
\end{prop}
\begin{proof}
For any $m,n\in\bN$, let $\mb{Res}_{m,n}^{m+n}:=\mb{Res}_{\mH\mC_{m,n}}^{\mH\mC_{m+n}}$
and $\varphi_{m,n}:=\varphi_m\ot\varphi_n$.
It needs to prove that there exists an isomorphism of functors from
$\mH\mC_{m+n}$-mod to $\mH\mC_{m,n}$-mod:
\beq\label{iso}\varphi_{m,n}\circ\mb{Res}_{m,n}^{m+n}\circ\varphi_{m+n}^{-1}
\cong\tau_{12}\circ\mb{Res}_{n,m}^{m+n}.
\eeq
By the definition of the Nakayama automorphisms $\varphi_n\,(n\in\bN)$ in (\ref{Nak}), we have
\[\varphi_{m+n}\circ\tilde{\mu}_{m,n}
=\tilde{\mu}_{n,m}\circ\varphi_{n,m}\circ f_{12},\]
where $f_{12}:\mH\mC_m\ot\mH\mC_n\rw\mH\mC_n\ot\mH\mC_m$ is the flip isomorphism (\ref{perm}). Hence, the LHS of (\ref{iso}) is
\[^{\tilde{\mu}_{m,n}\circ\varphi_{m,n}}(\mH\mC_{m+n})^{\varphi_{m+n}}
\ot_{\mH\mC_{m+n}}\mb{--}\cong
{^{\varphi_{m+n}\circ\tilde{\mu}_{n,m}\circ f_{12}}}(\mH\mC_{m+n})^{\varphi_{m+n}}
\ot_{\mH\mC_{m+n}}\mb{--}\cong
{^{\tilde{\mu}_{n,m}\circ f_{12}}}\mH\mC_{m+n}
\ot_{\mH\mC_{m+n}}\mb{--},\]
which is exactly the RHS of (\ref{iso}).
\end{proof}

\begin{prop}\label{comm}
For the tower of 0-Hecke-Clifford superalgebras, we have an isomorphism of functors $\mb{Res}\cong\tau_{12}\circ\mb{Res}$
 on $\mH\mC$-pmod.
\end{prop}
\begin{proof}
For any $m,n\in\bN$, let $\iota_{m,n}:=\iota_m\ot\iota_n$. First note that the following isomorphism of functors from $\mH_{m+n}\mb{-mod}$ to $\mH\mC_{m,n}\mb{-mod}$ holds:
\[{^{\tilde{\mu}_{m,n}}}(\mH\mC_{m+n})^{\iota_{m+n}}\ot_{\mH_{m+n}}\mb{--}
\cong(\mH\mC_{m,n})^{\iota_{m,n}}\ot_{\mH_{m,n}}
({^{\mu_{m,n}}}(\mH_{m+n})^{\iota_{m+n}}\ot_{\mH_{m+n}}\mb{--}).\]
In particular, for $\al\vDash m+n$, we choose $P_\al\in \mH_{m+n}\mb{-pmod}$ to get that
\beq\label{res}
\mb{Res}_{\mH\mC_{m,n}}^{\mH\mC_{m+n}}\tilde{P}_\al\cong {^{\tilde{\mu}_{m,n}}}(\mH\mC_{m+n})^{\iota_{m+n}}\ot_{\mH_{m+n}}P_\al\cong
(\mH\mC_{m,n})^{\iota_{m,n}}\ot_{\mH_{m,n}}
\mb{Res}_{\mH_{m,n}}^{\mH_{m+n}}P_\al.\eeq
Now by Lemma \ref{twist}, we have
\[\begin{split}
^{\tau_{12}}\mb{Res}_{\mH\mC_{m,n}}^{\mH\mC_{m+n}}\tilde{P}_\al&\cong
{^{\tau_{12}}}(\mH\mC_{m,n})^{\iota_{m,n}}\ot_{\mH_{m,n}}
\mb{Res}_{\mH_{m,n}}^{\mH_{m+n}}P_\al\\
&\cong
(\mH\mC_{n,m})^{\tau_{12}\circ\iota_{m,n}}\ot_{\mH_{m,n}}
\mb{Res}_{\mH_{m,n}}^{\mH_{m+n}}P_\al\\
&\cong(\mH\mC_{n,m})^{\iota_{n,m}}\ot_{\mH_{n,m}}
{^{\tau_{12}}}\mb{Res}_{\mH_{m,n}}^{\mH_{m+n}}P_\al.
\end{split}\]
As NSym is cocommutative, so is $\mK(\mH)$ by the Frobenius map (\ref{fro}), i.e. $\mb{Res}\cong\tau_{12}\circ\mb{Res}$
on $\mH$-pmod, thus it also holds on $\mH\mC$-pmod by the above discussion. In fact, if write $[\mb{Res} P_\al]=\sum c_{\al_1,\al_2}^\al[P_{\al_1}]\ot[P_{\al_2}]$ (The explicit formula is described by the shuffle product, see \cite{MR},\cite[(16)]{TU}), then by (\ref{res}) we simply have
$[\mb{Res}\tilde{P}_\al]=\sum c_{\al_1,\al_2}^\al[\tilde{P}_{\al_1}]\ot[\tilde{P}_{\al_2}]$.
\end{proof}

By Prop. \ref{dual}, \ref{frob}, \ref{flip}, \ref{comm}, we finally conclude that $(\widetilde{\mK},\,\widetilde{\mG})$ forms a dual pair of graded Hopf algebras with respect to the pair (\ref{pa1}). That is, for $P,Q\in\mH\mC\mb{-pmod},M,N\in\mH\mC\mb{-mod}$,
\[\begin{array}{l}
\lan\mb{Ind}([P]\ot[Q]),[M]\ran=\lan[P]\ot[Q],\mb{Res}[M]\ran,\\
\lan[P],\mb{Ind}([M]\ot[N])\ran=\lan\mb{Res}[P],[M]\ot[N]\ran.
\end{array}\]

It is easy to see that there exists an algebra involution $\bar{\varphi}$ of $\mH_n$ defined by
\[\bar{\varphi}(T_i)=T_{n-i},\,i=1,\dots,n-1.\]
$\bar{\varphi}$ is also the Nakayama automorphism of $\mH_n$ \cite[Lemma 4.2]{SY}. Now we have the following result
\begin{prop}\label{na}
For any $\al\vDash n$,
\[{^{\bar{\varphi}}}S_\al\cong S_{\bar{\al}},\,{^{\bar{\varphi}}}P_\al\cong P_{\bar{\al}}.\]
\end{prop}
\begin{proof}
For the twisted module ${^{\bar{\varphi}}}S_\al$, we have
\[T_i._{\bar{\varphi}}\eta_\al=T_{n-i}.\eta_\al=
\begin{cases}
-\eta_\alpha,& i\in D(\bar{\alpha}),\\
0,&\mbox{otherwise},
\end{cases}\]
which implies the first isomorphism.

For the second one, we note that for any $w=w_1\cdots w_n\in\fs_n$,
\[w\in\fd_\al\Leftrightarrow w_i>w_{i+1},\,i\in D(\al)
\Leftrightarrow w_0ww_0(i)>w_0ww_0(i+1),\,i\in D(\bar{\al})
\Leftrightarrow w_0ww_0\in\fd_{\bar{\al}}.\]
That is, $w_0\fd_\al w_0=\fd_{\bar{\al}}$ and also $i\in D(w)\Leftrightarrow n-i\in D(w_0ww_0)$. Hence, by the module structure (\ref{act}) of $P_\al$, we know that for any $w\in\fd_\al$,
\[\begin{split}
T_i._{\bar{\varphi}}u_w&=T_{n-i}.u_w
=\begin{cases}
-u_w, &n-i\in D(w^{-1}),\\
u_{s_{n-i}w}, &n-i\notin D(w^{-1}), s_{n-i}w\in\fd_\al,\\
0,&\mbox{otherwise}.
\end{cases}\\
&=\begin{cases}
-u_w, &i\in D((w_0ww_0)^{-1}),\\
u_{s_{n-i}w}, &i\notin D((w_0ww_0)^{-1}), w_0s_{n-i}ww_0=s_iw_0ww_0\in\fd_{\bar{\al}},\\
0,&\mbox{otherwise}.
\end{cases}
\end{split}
\]
Hence, $u_w\mapsto u'_{w_0ww_0},\,w\in\fd_\al$ gives the second isomorphism, where $\{u'_w:w\in\fd_{\bar{\al}}\}$ is the standard basis of $P_{\bar{\al}}$.
\end{proof}

\section{From 0-Hecke-Clifford algebras to the peak algebra of symmetric groups}
In this section we inherit the known result in \cite{BHT} to clarify more explicitly the relation between $(\mb{Peak},\mb{Peak}^*)$ and the supermodule categories of 0-Hecke-Clifford algebras, especially the dual Hopf pair $(\widetilde{\mK},\,\widetilde{\mG})$ discussed in the previous section. Then we consider the corresponding Heisenberg double in order to prove the freeness of $\mb{Peak}^*$ over $\Om$.

\subsection{From $\widetilde{\mK}$ to $\mb{Peak}$}

First of all, we define the adjoint map
of the Frobenius isomorphism $\widetilde{\mb{Ch}}$ relating two non-degenerate pairs $[\cdot,\cdot]$ in (\ref{pp}) and $\lan\cdot,\cdot\ran$ in (\ref{pa1}). That is a Hopf isomorphism $\widetilde{\mb{Ch}}^*:\mb{Peak}\rightarrow\widetilde{\mK}$
satisfying
\[[F,\widetilde{\mb{Ch}}([M])]=\langle\widetilde{\mb{Ch}}^*(F),[M]\rangle,\,
F\in\mb{Peak},[M]\in\tilde{\mG}.\]

The following result provides the explicit form of $\widetilde{\mb{Ch}}^*$:
\begin{theorem}\label{adj}
For any composition $\al$, we have
\beq\label{prj}\widetilde{\mb{Ch}}^*(\Te(R_\al))=[\tilde{P}_\al].\eeq
Moreover, the following commutative diagram of Hopf algebras holds:
\[\xymatrix@=2em{
\mathcal{K}\ar@{->}[r]^-{\mbox{\scriptsize Ind}_{\mH}^{\mH\mC}}
&\widetilde{\mK}\ar@{->}[r]^-{\widetilde{\chi}}
&\widetilde{\mG}
\ar@{->}[d]^-{\widetilde{\mb{\scriptsize Ch}}}\\
\mbox{NSym}\ar@{->}[r]^-{\Theta}\ar@{->}[u]_-{\mbox{\scriptsize Ch}^*}
&\mbox{Peak}\ar@{->}[u]_-{\widetilde{\mb{\scriptsize Ch}}^*}\ar@{->}[r]^-{\pi}&\mb{Peak}^*}\]
where $\widetilde{\chi}:\widetilde{\mK}\rw\widetilde{\mG}$ is the Cartan map of 0-Hecke-Clifford algebras. In particular, $\mb{Im }\chi=\Om$.
\end{theorem}
\begin{proof}
Given compositions $\al,\be$, we have
\[\begin{split}
\lan\widetilde{\mb{Ch}}^*(\Te(R_\al)),[\tilde{S}_\be]\ran&=
[\Te(R_\al),\widetilde{\mb{Ch}}([\tilde{S}_\be])]
=\lan R_\al,\mb{Ch}\lb[\mb{Res}_{\mH}^{\mH\mC}\tilde{S}_\be]\rb\ran\\
&=\lan \mb{Ch}^*(R_\al),[\mb{Res}_{\mH}^{\mH\mC}\tilde{S}_\be]\ran
=\lan[P_\al],[\mb{Res}_{\mH}^{\mH\mC}\tilde{S}_\be]\ran\\
&=\lan[\mb{Ind}_{\mH}^{\mH\mC}P_\al],[\tilde{S}_\be]\ran
=\lan[\tilde{P}_\al],[\tilde{S}_\be]\ran,
\end{split}\]
where the second equality is due to (\ref{pro}) and (\ref{emb}), and
the second last one bases on the fact that induction functor is left adjoint to restriction. Since the pair $\lan\cdot,\cdot\ran$ is non-degenerate, we get the desired formula, equivalent to the left commutative square.

For the right commutative square, we only need to prove that
\[\widetilde{\mb{Ch}}\circ\widetilde{\chi}([\tilde{P}_\al])
=\pi\circ\Te(R_\al)\]
by formula (\ref{prj}). The module structure (\ref{act}) of $P_\al\,(\al\vDash n)$ and Lemma \ref{ord} imply that if fix a total order $\prec$ refining the Bruhat order on $\fd_\al$, then $\tilde{P}_\al$  has a super submodule filtration $\{\tilde{P}_\al^w\}_{w\in\fd_\al}$, where $\tilde{P}_\al^w:=\{c_Du_z:D\subseteq[n],z\in\fd_\al,w\preceq z\}$.
Also, the subquotient $\tilde{P}_\al^w/\sum_{w\prec z}\tilde{P}_\al^z\cong\tilde{S}_{c(w^{-1})}$ for any $w\in\fd_\al$.
Hence,
\[\begin{split}
\widetilde{\mb{Ch}}\circ\widetilde{\chi}([\tilde{P}_\al])
&=\sum_{w\in\fd_\al}\widetilde{\mb{Ch}}(\tilde{S}_{c(w^{-1})})
=\sum_{w\in\fd_\al}K_{P(w^{-1})}\\
&=\sum_\be|\{w\in\fs_n:w\in\fd_\al,\,w^{-1}\in\fd_\be\}|K_{P(\be)}
=\sum_\be\lan R_\be,r_\al\ran K_{P(\be)}\\
&=\sum_P\lan \sum_{P(\be)=P}R_\be,\pi(R_\al)\ran K_P
=\sum_P\lan\Xi_P,r_\al\ran K_P\stackrel{(\ref{pro})}{=}\sum_P[\Xi_P,\vt(r_\al)]K_P\\
&=\vt(r_\al)\stackrel{(\ref{co})}{=}\te\circ\pi(R_\al)
\stackrel{(\ref{co})}{=}\pi\circ\Te(R_\al),
\end{split}\]
where the fourth equality is due to the following well-known formula
of Gessel (see \cite[Prop. 5.9]{KT}):
\[\lan R_\be,r_\al\ran=\lan R_\al,r_\be\ran=
|\{w\in\fs_n:w\in\fd_\al,\,w^{-1}\in\fd_\be\}|.\]
\end{proof}

\begin{cor}
For any composition $\al$, we have the following decomposition formula:
\beq\label{decp}\tilde{P}_\al\cong\bigoplus_{P(\be)\subseteq D(\al)\triangle(D(\al)+1)}{HClP_{P(\be)}}^{\oplus 2^{l_\be}},\,l_\be:=\left\lfloor\tfrac{|P(\be)|+1}{2}\right\rfloor,\eeq
where we use $HClP_P$ to denote the projective cover of $HClS_P$ for any peak set $P$.
\end{cor}
\begin{proof}
From \cite[(6)]{BHT} we know that
\beq\label{rx}\Te(R_\al)=\sum_{P\subseteq D(\al)\triangle(D(\al)+1)}2^{|P|+1}\Xi_P.\eeq
Now by Theorem \ref{adj},
\[\begin{split}
\lan[\tilde{P}_\al],[\tilde{S}_\be]\ran&=
[\Te(R_\al),\widetilde{\mb{Ch}}([\tilde{S}_\be])]
=[\Te(R_\al),K_{P(\be)}]\\
&=\begin{cases}
2^{|P(\be)|+1},&\mb{if }P(\be)\subseteq D(\al)\triangle(D(\al)+1),\\
0,& \mb{otherwise}.
\end{cases}
\end{split}\]
Combining it with (\ref{eva}), (\ref{dec}) and Prop. \ref{type}, we get the desired decomposition of $\tilde{P}_\al$.
\end{proof}
\noindent
In particular, the generator $Q_n=\Te(R_{(n)})$ corresponds to the projective simple supermodule $[\tilde{P}_{(n)}]=[\tilde{S}_{(n)}]=[HClS_{\emp_n}]=[HClP_{\emp_n}]$ via $\widetilde{\mb{Ch}}^*$ and
\[\lan[\tilde{P}_{(n)}],[\tilde{S}_\be]\ran=2\de_{P(\be),\emp},\]
also due to Schur's Lemma as the simple supermodule $HClS_{\emp_n}$ is of type Q. Note that in the classical case, $q_n\in\Om$ corresponds to the \textit{basic spin module} $Cl_n$ of the Sergeev algebra $Cl_n\rtimes\bk\fs_n$ under the Frobenius isomorphism \cite[\S 3.3]{CW}.

Now we abuse the notation to denote by $\bar{\varphi}$ the Hopf algebra anti-involution of NSym such that
$\bar{\varphi}(H_n)=H_n\,(n\in\bN)$. Note that $\bar{\varphi}(R_\al)=R_{\bar{\al}}$.
Meanwhile, since $\bar{\varphi}(\mb{Ker }\Te)=\mb{Ker }\Te$, $\bar{\varphi}$ induces a Hopf algebra anti-involution of Peak, which we abuse to denote by $\varphi$, then
$\varphi\circ\Te=\Te\circ\bar{\varphi}$.

We are in the position to prove the following restriction rule.
\begin{theorem}
For any composition $\al$, the following commutative diagram of Hopf algebras holds:
\[\xymatrix@=2em{
\widetilde{\mK}\ar@{->}[r]^-{\mb{\scriptsize Res}_\mH^{\mH\mC}}&\mK\\
\mb{Peak}\ar@{->}[r]^-{\bar{\varphi}\circ i\circ\varphi}\ar@{->}[u]_-{\widetilde{\mb{\scriptsize Ch}}^*}
&\mb{NSym}\ar@{->}[u]_-{\mb{\scriptsize Ch}^*}}\,\mb{\quad or\quad}\xymatrix@=2em{
\widetilde{\mK}\ar@{->}[r]^-{{^{\bar{\varphi}}}\mb{\scriptsize Res}_\mH^{\mH\mC}}&\mK\\
\mb{Peak}\ar@{->}[r]^-{i\circ\varphi}
\ar@{->}[u]_-{\widetilde{\mb{\scriptsize Ch}}^*}
&\mb{NSym}\ar@{->}[u]_-{\mb{\scriptsize Ch}^*}}\]
where $i:\mb{Peak}\rw\mb{NSym}$ is the natural inclusion. Equivalently, for any composition $\al$,
\beq\label{re}[\mb{Res}_\mH^{\mH\mC}\tilde{P}_{\al}]=\sum_{P(\be)\subseteq D(\bar{\al})\triangle(D(\bar{\al})+1)}2^{|P(\be)|+1}[P_{\bar{\be}}].\eeq
\end{theorem}

\begin{proof}
By Prop. \ref{na}, the above two diagrams are equivalent.
Since $\mb{Res}_\mH^{\mH\mC}:\widetilde{\mK}\rw\mK$
is easily checked to be a Hopf algebra homomorphism, using formula (\ref{rx}) we only need to prove for the generators $Q_n\,(n\in\bN)$ that \beq\label{rn}[\mb{Res}_\mH^{\mH\mC}\tilde{P}_{(n)}]
=\mb{Ch}^*(\bar{\varphi}\circ i\circ\varphi(Q_n))
\stackrel{(\ref{gen})}{=}2\sum_{P(\be)=\emp_n}[P_{\bar{\be}}]
=2\sum_{k=0}^{n-1}[P_{(n-k,1^k)}],\,n\in\bN.\eeq
For $\tilde{P}_{(n)}=\tilde{S}_{(n)}$,
$v_D:=c_D\eta_{(n)}\,(D\subseteq[n])$ form a basis. By (\ref{com})
\[T_i.v_D=\begin{cases}
-v_D+v_{D\backslash\{i,i+1\}},&\mb{if }i,i+1\in D,\\
-v_D+v_{(D\backslash\{i+1\})\cup\{i\}},&\mb{if }i\notin D,i+1\in D,\\
0,&\mb{otherwise}.
\end{cases}\]
Define a partial order $\lhd$ on $2^{[n]}$ such that
$D$ covers $D'$ if $i,i+1\notin D'$ and $D=D'\cup\{i,i+1\}$, or
$i\in D',i+1\notin D'$ and $D=(D'\backslash\{i\})\cup\{i+1\}$ for some $i\in[n-1]$. We denote such covering relation by $D'\lhd_i D$. For any $0\leq k\leq n-1$, take $D_{n,k}\subseteq[n]$ to be one of the two sets $\{n-k+1,\dots,n\}\backslash\{1\}$ and $\{1\}\cup\{n-k+1,\dots,n\}$ with odd cardinality. Define
\[v_{n,k}:=v_{D_{n,k}}+\sum_D\ep_Dv_D\in\tilde{P}_{(n)},\]
where the sum is over those $D\lhd D_{n,k}$ such that $D\lhd_{i_1}\cdots\lhd_{i_r}D_{n,k}$ for some $i_1,\dots,i_r\in[n-k,n-1]$, and $\ep_D$ is the sign of length of any chain in $2^{[n]}$ from $D$ to $D_{n,k}$. Then
\[T_i.v_{n,k}=\begin{cases}
0,&\mb{if }i\in\{1,\dots,n-k-2\},\\
-v_{n,k},&\mb{if }i\in\{n-k,\dots,n-1\}.
\end{cases}\]
Meanwhile, $\fd_{(n-k,1^k)}=\{w_1\cdots w_n\in\fs_n:
w_1<\dots<w_{n-k}>w_{n-k+1}>\dots>w_n\}$ for $0\leq k\leq n-1$,
and the projective $\mH_n$-module $P_{(n-k,1^k)}$ is generated by $u_{1\cdots(n-k-1)n(n-1)\cdots(n-k)}$.
Now one can check that the homogeneous component $(\tilde{P}_{(n)})_{\bar{1}}=\bigoplus_{k=0}^{n-1}\mH_n.v_{n,k}$ and the isomorphism
$\mH_n.v_{n,k}\cong P_{(n-k,1^k)},\,v_{n,k}\mapsto u_{1\cdots(n-k-1)n(n-1)\cdots(n-k)}$ holds. Define $D'_{n,k}$ analogous to $D_{n,k}$ but with even cardinality and then $v'_{n,k}$ analogous to $v_{n,k}$. Then $(\tilde{P}_{(n)})_{\bar{0}}=\bigoplus_{k=0}^{n-1}\mH_n.v'_{n,k}$ and
$\mH_n.v'_{n,k}\cong P_{(n-k,1^k)}$ as in the odd case. Finally, we prove formula (\ref{rn}) and thus the commutative diagrams.

In general, by formula (\ref{rx})
\[\Te(R_\al)=\sum_{P(\be)\subseteq D(\al)\triangle(D(\al)+1)}2^{|P(\be)|+1}R_\be.\]
Then by (\ref{prj}) and the commutative diagram, we get the desired restriction rule (\ref{re}).
\end{proof}

\begin{exam}
For $n=5,k=2$, the diagram of $P_{(3,1^2)}$ is as follows.
\[\xymatrix@R=0.5em{
\raisebox{1.5em}{$\stackrel{\quad\rc -T_3,-T_4}{u_{12543}}$}\ar@{->}[r]^-{T_2}&
\raisebox{1.5em}{$\stackrel{\quad\rc -T_2,-T_4}{u_{13542}}$}\ar@{->}[r]^-{T_1}
\ar@{->}[dr]^-{T_3}&\raisebox{1.5em}{$\stackrel{\quad\rc -T_1,-T_4}{u_{23541}}$}
\ar@{->}[r]^-{T_3}&\raisebox{1.5em}{$\stackrel{\quad\rc -T_1,-T_3}{u_{24531}}$}\ar@{->}[r]^-{T_2}&
\raisebox{1.5em}{$\stackrel{\quad\rc -T_1,-T_2}{u_{34521}}$}\\
&&\quad\quad {u_{14532}}_{\rc -T_2,-T_3}\ar@{->}[ur]^-{T_1}&&}
\]
Now $D_{5,2}=\{1,4,5\}$ and
$v_{5,2}=v_{\{1,4,5\}}-v_{\{1,3,5\}}-v_{\{1\}}+v_{\{1,3,4\}}$.
It is  straightforward to check the isomorphism $\mH_5.v_{5,2}\cong P_{(3,1^2)},\,v_{5,2}\mapsto u_{12543}$.
\end{exam}

Next we give another kind of restriction rule for the induced projective modules. Identifying $\mH_{n-1}$ (resp. $\mH\mC_{n-1}$) with $\mH_{n-1,1}$ (resp. $\mH\mC_{n-1,1}$), we have the embedding  $\mu_n:\mH_{n-1}\rw\mH_n$ (resp. $\tilde{\mu}_n:\mH\mC_{n-1}\rw\mH\mC_n$)  defined from $\mu_{n-1,1}$ (resp. $\tilde{\mu}_{n-1,1}$).
\begin{theorem}
For any $\al=(\al_1,\dots,\al_r)\vDash n$, we have
\beq\label{re1}
{^{\tilde{\mu}_n}\tilde{P}_\al}\cong\lb\bigoplus_{1\leq i\leq r\atop \al_i>1}
\tilde{P}_{\al_{(i)}}\rb^{\oplus2}
\oplus\lb\bigoplus_{1\leq i\leq r-1\atop \al_i>1}
\tilde{P}_{\al'_{(i)}}\rb^{\oplus2}
\oplus{\tilde{P}_{(\al_2,\dots,\al_r)}}{^{\oplus2\de_{\al_1,1}}},
\eeq
where $\al_{(i)}:=(\al_1,\dots,\al_{i-1},\al_i-1,\al_{i+1},\dots,\al_r),\,
\al'_{(i)}=(\al_1,\dots,\al_{i-1},\al_i+\al_{i+1}-1,\al_{i+2},\dots,\al_r)$.
\end{theorem}
\begin{proof}
First note that we have the following isomorphism of functors on $\mH_n$-mod.
\[\begin{array}{l}
{^{\tilde{\mu}_n}(\mH\mC_n)^{\iota_n}}\ot_{\mH_n}\mb{--}
\cong\lb(\mH\mC_{n-1})^{\iota_{n-1}}\ot_{\mH_{n-1}}{^{\mu_n}}
(\mH_n)\ot_{\mH_n}\mb{--}\rb^{\oplus 2}\\
c_DT_w\ot\mb{--}\mapsto (c_D\ot T_w\ot\mb{--},0),\,
c_Dc_nT_w\ot\mb{--}\mapsto (0,c_D\ot T_w\ot\mb{--})
\end{array}
\]
for any $D\subseteq[n-1],w\in\fs_n$. In particular,
\[{^{\tilde{\mu}_n}\tilde{P}_\al}
={^{\tilde{\mu}_n}(\mH\mC_n)^{\iota_n}}\ot_{\mH_n}P_\al
\cong\lb(\mH\mC_{n-1})^{\iota_{n-1}}\ot_{\mH_{n-1}}
{^{\mu_n}P_\al}\rb^{\oplus 2},\]
which reduces formula (\ref{re1}) to
\[{^{\mu_n}P_\al}\cong\lb\bigoplus_{1\leq i\leq r\atop \al_i>1}
P_{\al_{(i)}}\rb
\oplus\lb\bigoplus_{1\leq i\leq r-1\atop \al_i>1}
P_{\al'_{(i)}}\rb
\oplus{P_{(\al_2,\dots,\al_r)}}^{\oplus\de_{\al_1,1}}.\]
For example, ${^{\mu_5}P_{(1,2,2)}}\cong P_{(2,2)}\oplus P_{(1^2,2)}
\oplus P_{(1,3)}\oplus P_{(1,2,1)}$.

Recall that $P_\al$ has basis $\{u_w:w\in\fd_\al\}$ with module action (\ref{act}). Let $k_i:=\al_1+\cdots+\al_i,\,i=1,\dots,r$.
It is easy to see that the $j$'s such that $w_j=n$ for some $w=w_1\cdots w_n\in\fd_\al$ are those satisfying $j=k_i$ with $i=1$ or $1<i\leq r,\al_i>1$. Now for any such $j=k_i$, denote $\fd_\al^{(i)}:=\{w\in\fd_\al:w_j=n\}$ and $P^{(i)}_\al:=\mb{span}_\bk\{u_w:w\in\fd_\al^{(i)}\}$. We need to deal with three cases.

If $1<j=k_i<n,\al_i>1$, then $\{u_w:w\in\fd^{(i)}_\al,w_{j-1}>w_{j+1}\}$ spans the projective $\mH_{n-1}$-module $P_{\al_{(i)}}$. Modulo these vectors, the quotient space of $P^{(i)}_\al$ spanned by $\{u_w:w\in\fd^{(i)}_\al,w_{j-1}<w_{j+1}\}$ is isomorphic to another projective $\mH_{n-1}$-module $P_{\al'_{(i)}}$. If $j=k_1=\al_1=1$, then $P^{(1)}_\al$ spans $P_{(\al_2,\dots,\al_r)}$. If $j=k_r=n,\al_r>1$, then $P^{(r)}_\al\cong P_{\al_{(r)}}:=P_{(\al_1,\dots,\al_{r-1},\al_r-1)}$. In summary, we get the desired formula.
\end{proof}

\subsection{Application: $\mb{Peak}^*$ is free over $\Om$}
Finally we consider the Heisenberg double arising from $(\widetilde{\mK},\widetilde{\mG})$ in order to prove that $\mb{Peak}^*$ is a free $\Om$-module using the method of Savage et al. in \cite{SY}. In general, given a graded Hopf pair $(\mK(A),\mG(A))$, there exists a left action of $\mK(A)$ on $\mG(A)$ such that $\mG(A)$ is a $\mK(A)$-module algebra. It is defined by
\beq\label{ac}[P].[M]:=\begin{cases}\left[\mb{Hom}_{A_i}\lb P,{^{\ga_{n-i,i}}}\mb{Res}_{A_{n-i}\ot A_i}^{A_n}M\rb\right],
&P\in A_i\mb{-pmod},\,M\in A_n\mb{-mod},\,i\leq n,\\
0,&\mb{otherwise},
\end{cases}\eeq
where $\ga_{n-i,i}:A_i\rw A_{n-i}\ot A_i,\,a\mapsto1_{n-i}\ot a$ is the natural embedding. Note that for the Cartan map $\chi$, $\mG_{\mb{\scriptsize proj}}(A):=\mb{Im }\chi$ is stable under such action, thus a submodule of $\mG(A)$. Now one can define the following two kinds of \textit{Heisenberg doubles}:
\beq\label{hd}\fh(A):=\mG(A)\#\mK(A),\,
\fh_{\bs{proj}}(A):=\mG_{\bs{proj}}(A)\#\mK(A).\eeq
The notation $\#$ means smash product construction on $H\ot B$ from a Hopf algebra $H$ and an $H$-module algebra $B$.

Let
\[H^-:=\mK(A),\,H^+:=\mG(A),\,
H^+_{\bs{proj}}:=\mG_{\bs{proj}}(A),\,\fh:=\fh(A),\,
\fh_{\bs{proj}}:=\fh_{\bs{proj}}(A).\]
Then $H^+$ becomes a left $\fh$-module, called the \textit{lowest weight Fock representation}, where $H^+$ acts by left multiplication and $H^-$ acts by formula (\ref{ac}). For any $\fh$-module $V$, $v\in V$ is called a \textit{lowest weight vacuum vector} if $H^-v=0$.


From now on, we focus on the case when the tower $A=\mH\mC$.
\begin{lem}\label{lwv}
Suppose $V$ is an $\fh_{\bs{proj}}$-module generated by a finite set of lowest weight vacuum vectors. Then $V$ is a direct sum of lowest weight Fock representations.
\end{lem}
\begin{proof}
By the Stone-von Neumann Theorem for Heisenberg doubles \cite[Theorem 2.11]{SY} and Theorem \ref{adj}, we have $\fh_{\bs{proj}}\cdot v\cong H^+_{\bs{proj}}\cong\Om$ as an irreducible $\fh_{\bs{proj}}$-module over $\bk$ for any lowest weight vacuum vector $v\in V$. Now the same argument in \cite[Lemma 9.1]{SY} gives the complete reducibility of $V$.
\end{proof}

For any $\be=(\be_1,\dots,\be_r),\,\ga=(\ga_1,\dots,\ga_s)$, let
$\be\cdot\ga:=(\be_1,\dots,\be_r,\ga_1,\dots,\ga_s)$,
then $\De(M_\al)=\sum_{\be\cdot\ga=\al}M_\be\ot M_\ga$.
Since $\vt$ is a Hopf algebra epimorphism, $N_\al:=\vt(M_\al)\,(\al\in\mC)$ span the Stembridge algebra $\mb{Peak}^*$ and
\[\De(N_\al)=\sum_{\be\cdot\ga=\al}N_\be\ot N_\ga.\]

Define an increasing filtration of $\fh_{\bs{proj}}$-submodules of $\mb{Peak}^*$ as follows. For $n\in\bN_0$, let
\[(\mb{Peak}^*)^{(n)}:=\sum_{\ell(\al)\leq n}\fh_{\bs{proj}}\cdot N_\al.\]
In particular, $(\mb{Peak}^*)^{(0)}:=\Om$ and by convention we also let $(\mb{Peak}^*)^{(-1)}:=0$.

\begin{prop}
The space $\mb{Peak}^*$ of peak quasisymmetric functions is free as an $\Om$-module.
\end{prop}
\begin{proof}
For any composition $\al$ such that $\ell(\al)=n$, by the grading argument of $[\cdot,\cdot]$,
\[Q_m.N_\al=\sum_{\be\cdot\ga=\al}[Q_m,N_\ga]
N_\be\in(\mb{Peak}^*)^{(n-1)},\,m\in\bN.\]
Hence, in the quotient
$V_n:=(\mb{Peak}^*)^{(n)}/(\mb{Peak}^*)^{(n-1)}$, such $N_\al$'s are lowest weight vacuum vectors. Clearly these vectors generate $V_n$ and thus by Lemma \ref{lwv},
\[V_n=\bigoplus_{v\in L_n}\Om.v,\]
where $L_n$ is some collection of vacuum vectors in $V_n$.

Consider the short exact sequence of $\Om$-modules
\[0\rw(\mb{Peak}^*)^{(n-1)}\rw(\mb{Peak}^*)^{(n)}\rw V_n\rw0.\]
Since $V_n$ is a free $\Om$-module, the above sequence splits. Now $(\mb{Peak}^*)^{(0)}\cong\Om$, so we know that all $(\mb{Peak}^*)^{(n)}\,(n\in\bN_0)$
are free over $\Om$ by induction on $n$. It means that we can choose nested sets of vectors in $\mb{Peak}^*$
\[\tilde{L}_0\subseteq\tilde{L}_1\subseteq\tilde{L}_2\subseteq\cdots\]
such that, for any $n\in\bN_0$, $(\mb{Peak}^*)^{(n)}=\bigoplus_{\tilde{v}\in\tilde{L}_n}\Om.\tilde{v}$.
Let $\tilde{L}=\bigcup_{n\in\bN_0}\tilde{L}_n$. Then $\mb{Peak}^*=\bigoplus_{\tilde{v}\in\tilde{L}}\Om.\tilde{v}$.
\end{proof}

\centerline{\bf Acknowledgments}
We would like to thank N. Jing and A. Savage for valuable discussions during the workshop at SCUT, January 2015.


\bigskip
\bibliographystyle{amsalpha}

\end{document}